\documentclass[a4paper,12pt]{amsart}
\usepackage{graphicx}
\usepackage{tikz-cd}
\usepackage{amssymb}
\usepackage{upquote}
\usepackage{mathtools}
\usepackage{amsthm}
\usepackage[normalem]{ulem}
\usetikzlibrary{arrows}
\usepackage{color}
\usepackage{multirow}
\usepackage{enumitem}
\usepackage{textcomp}
\usepackage{xfrac}
\usepackage{hyperref, aliascnt}
\usepackage[capitalise, nameinlink, noabbrev, nosort]{cleveref}
\usepackage{extarrows}

\hypersetup{pdfauthor={Kristin Courtney, Wilhelm Winter}, pdftitle={Nuclearity and CPC*-Systems}, colorlinks=true, linkcolor=darkgray, linkbordercolor=darkgray, citecolor=darkgray, citebordercolor=darkgray, linktocpage=true}

\newtheorem{theorem}{Theorem}[section]

\newaliascnt{lemma}{theorem}
\newtheorem{lemma}[lemma]{Lemma}
\aliascntresetthe{lemma}

\newaliascnt{proposition}{theorem}
\newtheorem{proposition}[proposition]{Proposition}
\aliascntresetthe{proposition}

\newaliascnt{corollary}{theorem}
\newtheorem{corollary}[corollary]{Corollary}
\aliascntresetthe{corollary}

\newaliascnt{claim}{theorem}

\aliascntresetthe{claim}

\newtheorem*{theorem*}{Theorem}

\theoremstyle{definition}

\newaliascnt{definition}{theorem}
\newtheorem{definition}[definition]{Definition}
\aliascntresetthe{definition}

\newtheorem*{definition*}{Definition}

\newaliascnt{notation}{theorem}

\aliascntresetthe{notation}

\newaliascnt{remark}{theorem}
\newtheorem{remark}[remark]{Remark}
\aliascntresetthe{remark}

\newaliascnt{remarks}{theorem}
\newtheorem{remarks}[remarks]{Remarks}
\aliascntresetthe{remarks}

\newaliascnt{example}{theorem}

\aliascntresetthe{example}

\newaliascnt{examples}{theorem}

\aliascntresetthe{examples}

\newtheorem*{corollary*}{Corollary}
\newtheorem*{proposition*}{Proposition}
\newtheorem*{remark*}{Remark}

\numberwithin{equation}{section}

\newcommand{\sbt}{\,\begin{picture}(-1,1)(-1,-2.5)\circle*{3}\end{picture}\,\,\, }

\allowdisplaybreaks

\begin{document}
\def\C{\mathbb{C}}
\def\N{\mathbb{N}}
\def\M{\mathrm{M}}
\def\id{\mathrm{id}}
\def\Cstar{\mathrm{C}^*}
\def\CPCstar{\mathrm{CPC}^*}
\def\Lim{\overset{\tikz \draw [arrows = {->[harpoon]}] (-1,0) -- (0,0);}{(F_n,\rho_n)}_n}

\title{Nuclearity and $\CPCstar$-Systems}
 \author[K.\ Courtney]{Kristin Courtney}
     \address{Department of Mathematics and Computer Science,
University of Southern Denmark, Campusvej 55,
DK-5230 Odense M,
Denmark}
\email{kcourtney@imada.sdu.dk}
      \author[W.\ Winter]{Wilhelm Winter}
     \address{Mathematical Institute, University of M\"unster, Einsteinstr.\ 62, 48149 M\"unster, Germany}
     \email{wwinter@uni-muenster.de}
     \subjclass[2020]{46L05} 
 \keywords{Completely positive approximation, nuclearity, inductive limits, order zero maps, completely positive maps, NF systems, NF algebras}
 \thanks{
This research was supported by the Deutsche Forschungsgemeinschaft (DFG, German Research Foundation) under Germany's Excellence Strategy -- EXC 2044 -- 390685587, Mathematics M\"unster -- Dynamics -- Geometry -- Structure, the Deutsche Forschungsgemeinschaft (DFG, German Research Foundation) -- Project-ID 427320536 -- SFB 1442, and ERC Advanced Grant 834267 -- AMAREC
}
\date{\today}

\begin{abstract}
We write arbitrary separable nuclear $\Cstar$-algebras as
limits of inductive systems of finite-dimensional $\Cstar$-algebras with completely positive connecting maps. The characteristic feature of such $\CPCstar$-systems is that the maps become more and more orthogonality preserving. This condition makes it possible to equip the limit, a priori only an operator space, with a multiplication turning it into a $\Cstar$-algebra. Our concept generalizes the NF systems of Blackadar and Kirchberg beyond the quasidiagonal case.
\end{abstract}

\maketitle

\section*{Introduction}
\renewcommand*{\thetheorem}{\Alph{theorem}}
\renewcommand*{\thedefinition}{\Alph{definition}}
\renewcommand*{\theproposition}{\Alph{theorem}}

\noindent 
Nuclear $\Cstar$-algebras are those for which the identity map can be
approximated in point norm topology by completely positive contractions
through finite-dimensional $\Cstar$-algebras. More formally, a $\Cstar$-algebra $A$ has the completely positive approximation property if there are a net $(F_\lambda)_{\lambda\in \Lambda}$ of finite-dimensional $\Cstar$-algebras and completely positive contractive (c.p.c.)\ maps \[A\xlongrightarrow{\psi_\lambda}F_\lambda \xlongrightarrow{\varphi_\lambda} A\] such that $\varphi_\lambda \circ \psi_\lambda$ converges pointwise in norm to the identity map $\id_A$. For separable $\Cstar$-algebras
(the case we are mostly interested in in this paper) one can always 
restrict to systems of approximations with countable index sets.
Upon composing such approximations one obtains a system of the
form 
{\footnotesize
\[
\begin{tikzcd}[cramped]
\hdots &A\arrow{rr}{\id_A}\arrow{dr}{\psi_{n-1}} & & A\arrow{rr}{\id_A}\arrow{dr}{\psi_n} & & A\arrow{rr}{\id_A}\arrow{dr}{\psi_{n+1}} && A 
& 
\hdots 
\\
\hdots && F_{n-1}\arrow{ur}{\varphi_{n-1}}\arrow{rr}{\rho_{n,n-1}}&& F_n\arrow{ur}{\varphi_n}\arrow{rr}{\rho_{n+1,n}} && F_{n+1}\arrow{ur}{\varphi_{n+1}} && \hdots
\end{tikzcd}
\]
}
with commuting lower triangles, $\rho_{n,n-1}\coloneqq \psi_{n} \circ \varphi_{n-1}$, and approximately commuting upper triangles, i.e., $\varphi_n \circ \psi_n \longrightarrow \id_A$ in point-norm topology.
The sequence of upper triangles $(A, \varphi_n \circ \psi_n)_n$ may be regarded as an inductive system in the category of operator spaces with c.p.c.\ maps as morphisms,
and with the inductive limit again being an operator space, more precisely a self-adjoint 
subspace of the sequence algebra
$A_\infty = \prod_n A / \bigoplus_n A$, the quotient of bounded sequences
modulo null sequences in $A$. At least after passing to a subsystem (so that
the approximation errors become summable on larger and larger subsets), this inductive 
limit will agree with the embedding of $A$ into $A_\infty$ as constant sequences, and in particular it will indeed be a $\Cstar$-algebra. 
In this situation it is then reasonable to expect that the lower row 
\[
\begin{tikzcd}
\hdots & F_{n-1}\arrow{r}{\rho_{n,n-1}}& F_n\arrow{r}{\rho_{n+1,n}} & F_{n+1} & \hdots
\end{tikzcd}
\]
also encodes all pertinent information about $A$. This is a tempting 
point of view, since it would allow one to write any separable nuclear $\Cstar$-algebra 
as an inductive limit of finite-dimensional ones in a suitable category:
Such a limit, say $X$,  will be contained in the 
sequence algebra $F_\infty = \prod_n F_n / \bigoplus_n F_n$. It can be described as (the norm closure of) the union $\bigcup_n \rho_n(F_n)$, where $\rho_n\colon F_n \longrightarrow F_\infty$ denotes the limit map for each $n$, and there will be an induced c.p.c.\ isometry $\Psi\colon A \longrightarrow X$. The space $X$ will certainly reflect the linear structure of $A$ and also the order (even after passing to matrix amplifications), since the maps involved are completely positive. It will, however, in general not be a sub-$\Cstar$-algebra of $F_\infty$, and the map $\Psi$ will not be multiplicative. 

\bigskip

Studying operator algebras as inductive limits has a long history, spanning from Murray and von Neumann's work on 
hyperfiniteness in \cite{MvN} to semidiscreteness and Connes' celebrated classification 
of injective $\mathrm{II}_1$ factors (\cite{Con76}) on the von Neumann algebra side, and featuring 
the classification of UHF, AF, AT, and AH algebras on the $\Cstar$-algebra side (see \cite{Ell, Ell93, Gli60}). Inductive limits have been used to construct fundamental $\Cstar$-algebras with prescribed structural, $\mathrm{K}$-theoretic, and tracial data, such as the Jiang--Su Algebra $\mathcal{Z}$ \cite{JS99,RW10} and the Razak--Jacelon Algebra $\mathcal{W}$ \cite{Jac13}. More recently, inductive limits of Cartan pairs as described in \cite{BL20} were used 
in \cite{Li20} to show that all classifiable stably-finite $\Cstar$-algebras admit groupoid models. 
The inductive systems appearing in these results all have $^*$-homomorphisms as connecting maps. This entails that, at least on the $\Cstar$-algebra side, one has to admit more complicated building block algebras than just finite-dimensional ones in order to reach a reasonable bandwidth of examples. (Often one uses algebras built from bundles of finite-dimensional $\Cstar$-algebras over locally compact spaces, which poses limitations on the existence of sufficiently many connecting maps.)

\bigskip

In this paper we explore how to relax conditions on the connecting maps while keeping the building blocks finite-dimensional, hence as easy as possible. Our Ansatz is it to write an arbitrary
separable nuclear $\Cstar$-algebra as an inductive limit of finite-dimensional  $\Cstar$-algebras \emph{in the category 
of operator spaces}. In order to be conceptually 
satisfactory and to give access to applications, such an approach would
have to meet some minimal requirements: One needs to be able to recover
the $\Cstar$-algebra structure from the operator system $X$ (or from the 
inductive system, for that matter) in a practical way, and one 
needs to be able to concisely describe the limit operator spaces
arising in this manner. There are, however, two obstacles to overcome.

First, the connecting maps $\rho_{n,n-1}$ can in general not 
be arranged to be multiplicative, not even approximately: requiring 
multiplicativity would only grant access to AF algebras,
a small albeit interesting subclass of nuclear C*-algebras. In \cite{BK97},
Blackadar and Kirchberg suggested the notion of NF systems, 
which incorporates a certain type of asymptotic
multiplicativity and which covers the substantially larger class of NF algebras. 
While much broader than AF algebras, this class is still far 
from covering all nuclear $\Cstar$-algebras. 
Therefore, we are looking for a notion of inductive systems
of finite-dimensional $\Cstar$-algebras which is sufficiently flexible
in the sense that every separable nuclear $\Cstar$-algebra can be
written as the inductive limit of such a system, while at the same 
time it is rigid enough to encode not only 
the linear and matrix order structure, but also the multiplicative structure 
of an initially given $\Cstar$-algebra. 
It will be important that the algebraic
structure is not only \emph{encoded} in the system, but that it can actually 
be \emph{accessed} and extracted in a meaningful and sufficiently concrete way.

Second, assuming one has a notion of inductive systems as above, 
we need to distinguish them from arbitrary inductive systems of finite-dimensional $\Cstar$-algebras in the category of 
operator spaces. Ideally we will want to derive an \emph{intrinsic} 
characterization of the respective limits. (Theoretically one 
could compare the limit to \emph{every} $\Cstar$-algebra and check for 
the existence of an isomorphism in a suitable sense, but such an 
\emph{extrinsic} characterization would be unsatisfactory both conceptually
and practically, since it would largely obscure any information 
related to the finite-dimensional approximations.) 

The aforementioned notion of NF algebras as introduced by Blackadar and Kirchberg addresses the first obstacle,
and reveals a subtle facet of the second.

\begin{definition*}[\cite{BK97}]
 A system $(F_n,\rho_{n+1,n})_n$ consisting of a sequence $(F_n)_n$ of finite-dimensional $\Cstar$-algebras together with c.p.c.\ connecting maps $\rho_{n+1,n}\colon F_n\longrightarrow F_{n+1}$ is called an \emph{NF system}  if it is asymptotically multiplicative in the sense that for any $\varepsilon>0$, $k\geq 0$, and $x,y\in F_k$, there exists an $M>k$ so that for all $m>n>M$ we have 
    \begin{align}\label{defn*}
    \|\rho_{m,n}\big(\rho_{n,k}(x)\rho_{n,k}(y)\big)-\rho_{m,n}\big(\rho_{n,k}(x)\big)\rho_{m,n}\big(\rho_{n,k}(y)\big)\|<\varepsilon. \tag{a}
    \end{align}
 \end{definition*}

The limit of such a system is indeed a sub-$\Cstar$-algebra of $F_\infty$, 
with product given by
\begin{align}\label{defn**}
\rho_k(x)\rho_k(y)=\lim_n \rho_n\big(\rho_{n,k}(x)\rho_{n,k}(y)\big)\tag{b}
\end{align} 
for any $k\geq 0$ and $x,y\in F_k$. 
Therefore asymptotic multiplicativity of the involved connecting maps 
does encode the multiplicative structure of the limit, thus addressing 
the first obstacle above.

Blackadar and Kirchberg then characterize limits of NF systems,
so called NF algebras, as nuclear and quasidiagonal $\Cstar$-algebras.
While this is a large and highly relevant class of $\Cstar$-algebras, it is
not closed under taking quotients or extensions, which restricts the 
range of applications of NF systems considerably. This is subtly related to the fact that quasidiagonality is an \emph{external} approximation property, characterized extrinsically, in terms of comparison maps 
into other $\Cstar$-algebras.
 
\bigskip

We propose here a generalization of NF systems which 
permits to encode the multiplicative structure of the limit in
terms of the connecting maps between its finite dimensional 
approximands, yet renders access to all separable nuclear 
$\Cstar$-algebras, and therefore allows for an intrinsic 
characterization. As the key feature of our main definition we
ask the connecting maps to become more and more \emph{order zero}, thus replacing the asymptotic multiplicativity of
NF systems. We recall from \cite{WZ09} that, for a c.p.c.\ map, `order zero'  just means `orthogonality preserving', and that order zero c.p.c.\ maps remember a lot of the multiplicative structure through formulae of the type 
$\theta(x) \theta(yz) = \theta(xy)\theta(z)$. In case the domain is unital, one can characterize order zero maps via 
\begin{align}\label{dager}
\theta(1) \theta(yz) = \theta(y)\theta(z).\tag{c}
\end{align}

\begin{definition}[\cref{def: cpcstar}]
A system $(F_n,\rho_{n+1,n})_n$ consisting of a sequence $(F_n)_n$ of finite-dimensional $\Cstar$-algebras along with c.p.c.\ connecting maps $\rho_{n+1,n}\colon F_n\longrightarrow F_{n+1}$ is called a $\CPCstar$-\emph{system} if it is asymptotically order zero in the sense that for any $\varepsilon>0$, $k\geq 0$, and $x,y\in F_k$, there exists an $M>k$ so that for all $m>n,l>M$ we have 
    \begin{align}\label{2dager}
    \|\rho_{m,l}(1_{F_l})\rho_{m,n}\big(\rho_{n,k}(x)\rho_{n,k}(y)\big)-\rho_{m,n}\big(\rho_{n,k}(x)\big)\rho_{m,n}\big(\rho_{n,k}(y)\big)\|<\varepsilon.\tag{d}
    \end{align}
\end{definition}

Note that in view of \eqref{dager} the approximate identity \eqref{2dager} may indeed be viewed as an order zero version of \eqref{defn*}. For inductive limits of such systems, one may still describe a product as in \eqref{defn**}.

\begin{proposition}[\cref{thm: main 1}]
\label{propB}
Let $X\subset F_\infty$ be the limit of a  $\CPCstar$-system $(F_n,\rho_{n+1,n})_n$. 
Then there exists an associative bilinear map $\sbt\colon  X \times X \longrightarrow X$ satisfying
\begin{align*}
    \rho_k(x)\sbt \rho_k(y) & = \lim_n \rho_n(\rho_{n,k}(x) \rho_{n,k}(y))
\end{align*}
for all $k\geq 0$ and $x,y\in F_k$, so that $(X, \sbt)$ is a $\Cstar$-algebra with the involution and norm inherited as a subspace of $F_\infty$. 
\end{proposition}
We denote the $\Cstar$-algebra above by $\Cstar_{\sbt}(X)$ and call it the $\CPCstar$-limit of the system $(F_n,\rho_{n+1,n})_n$.

With a $\Cstar$-algebra structure established, the next question is: When is the limit nuclear? Thanks to Ozawa and Sato's ``one-way-CPAP" from \cite[Theorem 5.1]{Sato2019} (an approximation property which only involves the upwards maps $\varphi_n$), the answer is ``always". Remarkably, while $\CPCstar$-limits readily satisfy the one-way-CPAP, in order to arrive at nuclearity the proof of \cite[Theorem 5.1]{Sato2019} first confirms that $\Cstar_{\sbt}(X)^{**}$ is injective, which by Connes' theorem implies that $\Cstar_{\sbt}(X)$ is nuclear. Hence this result is non-constructive, meaning one does not build a completely positive approximation directly from the system $(F_n,\rho_{n+1,n})_n$. This is in line with Blackadar and Kirchberg's characterization of NF algebras as the nuclear and quasidiagonal $\Cstar$-algebras, which involves the fact that nuclearity passes to quotients, and hence also relies on Connes' theorem and is non-constructive in a similar way.

As a converse, we show that any system of c.p.c.\ approximations of a separable nuclear $\Cstar$-algebra $A$ with \emph{approximately order zero downwards maps} has a subsystem which gives rise to a $\CPCstar$-system with $\CPCstar$-limit isomorphic to $A$. It follows implicitly from \cite{BK97,WZ09} and explicitly from \cite{BCW} that every nuclear $\Cstar$-algebra \emph{does} admit such a system of approximations.

Combining these results, we can now describe all separable nuclear $\Cstar$-algebras as inductive limits of finite-dimensional $\Cstar$-algebras:
\begin{theorem}
\label{theoremC}
For a separable $\Cstar$-algebra $A$ the following are equivalent:
\begin{enumerate}[label=\textnormal{(\roman*)}]
    \item $A$ is nuclear.
    \item $A$ is $^*$-isomorphic to a $\CPCstar$-limit. 
\end{enumerate}
\end{theorem}

\cref{propB} shows how to extract the multiplicative structure from a $\CPCstar$-limit, thus addressing the first of the two issues raised above. As for the second, our description of $\CPCstar$-systems and the characterization of their limits is indeed an intrinsic one, not requiring any comparison with $\Cstar$-algebras outside. In upcoming work, we will explore how to apply this notion practically, in particular how to describe permanence properties of nuclearity, and how to access K-theory in terms of $\CPCstar$-systems.

\bigskip

The article is arranged as follows.
\cref{sect: preliminaries} establishes results on c.p.c.\ and order zero maps that will be used throughout. Some results in this section are well-known and some are new. In \cref{sect: cpc systems} we define $\CPCstar$-systems and prove (ii)$\Longrightarrow$(i) of \cref{theoremC}. \cref{sect: cpap} 
is dedicated to proving the reverse implication, and \cref{sect: NF} relates $\CPCstar$-systems to NF systems.

\bigskip

\noindent \text{\bf Acknowledgement:} We are grateful to Jamie Gabe for helpful comments on a draft of this paper and to the referee for their careful reading and helpful comments.

%Tristan Bice for a shortening of an approximation in \cref{lem: uniform order unit}(i).

\renewcommand*{\thetheorem}{\roman{theorem}}
\numberwithin{theorem}{section}
\renewcommand*{\thedefinition}{\roman{definition}}
\numberwithin{definition}{section}
\renewcommand*{\theproposition}{\roman{proposition}}
\numberwithin{proposition}{section}
\section{Order zero maps and order embeddings}\label{sect: preliminaries}

\noindent
In this section we highlight some properties of completely positive maps, with a particular focus on order zero maps and on complete order embeddings. 

We write $A_+$ for the cone of positive elements of a $\Cstar$-algebra $A$, $A^1$ and $A^1_+$ for the respective closed unit balls, $X' \cap A\subset M$ for the relative commutant in $A$ of a set $X$ (with $A$ and $X$ both contained in a larger $\Cstar$-algebra $M$), $CA = C_0((0,1],A)$ for the cone, $\M_r(A)$ for $r \times r$ matrices over $A$, and $\mathcal{M}(A)$ for the multiplier algebra of $A$ (regarded as a subalgebra of the double dual $A^{**}$).

\begin{definition}
Let $A$ and $B$ be $\Cstar$-algebras with self-adjoint subspaces $X\subset A$, $Y\subset B$, and $\theta \colon X\longrightarrow Y$ a linear map. We say $\theta$ is \emph{positive} if $\theta(x)\in Y\cap B_+$ for all $x\in X\cap A_+$. We say it is \emph{completely positive} (c.p.) if this holds for all matrix amplifications $\theta^{(r)}\colon \M_r(X)\longrightarrow \M_r(Y)$. 
We say $\theta$ is \emph{completely contractive} if 
$\sup_{r\geq 1} \|\theta^{(r)}\|\leq 1$.
A map which is completely positive and completely contractive is referred to as completely positive contractive (c.p.c.). 
\end{definition}

\begin{remark}\label{rmk: cpc is cc}
When $1_A\in X$ is a \emph{unital} self-adjoint subspace of a unital $\Cstar$-algebra $A$ (what is usually called an operator (sub)system), then for any c.p.\ map $\theta \colon X\longrightarrow Y$, we have 
\begin{align*}\label{eq: norm of cp map}
    \|\theta\|\leq  \sup_{r\geq 1}\|\theta^{(r)}\|\leq \|\theta(1_A)\|.
\end{align*}
 It follows that any completely positive and contractive map from a unital self-adjoint subspace of a $\Cstar$-algebra is automatically completely contractive. The respective statement holds when $A$ is a (not necessarily unital) $\Cstar$-algebra. 
\end{remark}

A particularly special class of completely positive maps are the orthogonality preserving or \emph{order zero} maps. At the end of this section, we will highlight several properties that order zero maps share with $^*$-homomorphisms and that may fail for general c.p.c.\ maps.

\begin{definition}
A c.p.\ map $\theta\colon A\longrightarrow B$ between $\Cstar$-algebras is \emph{order zero} if it maps orthogonal elements to orthogonal elements, or equivalently if for any $a,b\in A_+$, 
\[ab=0\ \Longrightarrow\ \theta(a)\theta(b)=0.\]
\end{definition}
Examples of c.p.\ order zero maps include $^*$-homomorphisms, but there exist c.p.\ order zero maps which are not multiplicative, such as the canonical embedding $\iota\colon A\longrightarrow CA$ of a $\Cstar$-algebra into its cone.   
Nonetheless, order zero maps are remarkably close to $^*$-homomorphisms, as is made evident by the structure theorem for c.p.\ order zero maps \cite[Theorem 3.3]{WZ09}. The following rendition combines elements of \cite[Theorem 3.3]{WZ09} and \cite[Proposition 3.2]{WZ09}.

\begin{theorem}[Structure theorem for order zero maps \cite{WZ09}]\label{thm: structure thm}
Let $A$ and $B$ be $\Cstar$-algebras, $\theta\colon A\longrightarrow B$ a c.p.\ order zero map, $C=\Cstar(\theta(A))$, $\{u_\lambda\}_\lambda$ an increasing approximate identity for $A$, and $h = \emph{s.o.\text{-}}\lim_\lambda \theta(u_\lambda)\\ \in C^{**}$. Then  $0\leq h\in \mathcal{M}(C)\cap C'\subset C^{**}$ 
with $\|h\|=\|\theta\|$, and there exists a $^*$-homomorphism $\pi_\theta\colon A\longrightarrow \mathcal{M}(C)$ so that for all $a\in A$, 
\[\theta(a)=h\pi_\theta(a)=h^{1/2}\pi_\theta(a)h^{1/2}.\]
\end{theorem}

\begin{remarks}\label{cor: order zero for unital}
(i) (\cite[Proposition 3.2]{WZ09}) When $A$ is unital, we have $h=\theta(1_A)$, and in general the map $\theta^{\sim}\colon A^{\sim}\longrightarrow \mathcal{M}(C)$ given by $\theta^{\sim}(a+\lambda 1_{A^\sim})=\theta(a)+\lambda h$ is the unique c.p.\ order zero extension of $\theta$. (By $A^{\sim}$ we denote the minimal unitization of $A$, so that $A^{\sim}=A$ when $A$ is unital.) In any case, $h\geq \theta(a)$ for all $a\in A_+^1$. Note that for $\theta^{\sim}$ to be order zero, one can replace $h$ with $1_{\mathcal{M}(C)}$ only if $\theta$ is in fact multiplicative, since u.c.p.\ order zero maps are automatically $^*$-homomorphisms. 

(ii) The structure theorem for order zero maps gives a useful characterization of when a c.p.\ map with unital domain is order zero: 
    Assume $A$ and $B$ are $\Cstar$-algebras with $A$ unital. Then a c.p.\ map $\theta\colon A\longrightarrow B$ is order zero if and only if $\theta(a)\theta(b)=\theta(1_A)\theta(ab)$ for all $a,b\in A$.

\end{remarks}

\begin{definition}
Let $A$ and $B$ be $\Cstar$-algebras and $X\subset A$, $Y\subset B$ two self-adjoint subspaces. We say a linear map $\theta \colon X\longrightarrow Y$ is a \emph{complete order embedding} if $\theta$ is c.p.\ and completely isometric 
with c.p.\ inverse $\theta^{-1}\colon \theta(X)\longrightarrow A$, i.e., for all $r\geq 1$, $\theta^{(r)}\colon \M_r(X)\longrightarrow \M_r(Y)$ is isometric and $x\in \M_r(X)\cap \M_r(A)_+ \Longleftrightarrow \theta^{(r)}(x)\in \M_r(Y)\cap \M_r(B)_+$ for all $x\in \M_r(X)$. 
A surjective complete order embedding is called a \emph{complete order isomorphism}. 
\end{definition}

\begin{remarks}\label{rmk: folklore 1}
(i) A complete order isomorphism between $\Cstar$-algebras is automatically 
a $^*$-isomorphism (cf.\ \cite[Theorem II.6.9.17]{Bla06}).

(ii) If $X=A$ and $\theta\colon A\longrightarrow Y\subset B$ is c.p.\ and completely isometric, then it is automatically a complete order embedding: 
If $\theta(a)\geq 0$ for some $0\neq a\in A^1$, then 
$a=a^*$ since $\theta$ is injective and $^*$-linear. If $a_-\neq 0$, then for 
$b\coloneqq\frac{\|a\|}{\|a_-\|}a_-+a_+$ 
we have
\begin{align*}
    \|a\|+\|a_-\|&=\|b-a\|\\
    &=\|\theta(b)-\theta(a)\|\\
    &\leq \| \|\theta(b)\|1_{B^\sim}-\theta(a)\|\\
    &=\|\|\theta(a)\|1_{B^\sim}-\theta(a)\|\\
    &\leq \|\theta(a)\|\\
    &=\|a\|,
\end{align*}
a contradiction, so $a$ was positive. The respective argument applies when $a$ lives in some matrix algebra over $A$, whence indeed the inverse of $\theta$ is completely positive. 

(iii) An isometric c.p.\ map $\theta\colon A\longrightarrow B$ between $\Cstar$-algebras which is \emph{not} surjective and \emph{not} unital may fail to be a complete order embedding -- even if it has a c.p.\ inverse $\theta^{-1}\colon \theta(A)\longrightarrow A$. 
We will work out such an example in forthcoming work.
The proposition below shows that for all order zero maps, this issue does not occur. We will use this in \cref{thm: order zero coi}.
\end{remarks}

\begin{proposition}\label{prop: isom perp is coe}
Suppose $\theta\colon A\longrightarrow B$ is a c.p.\ order zero map between $\Cstar$-algebras $A$ and $B$. Then $\theta$ is isometric if and only if it is a complete order embedding. 
\end{proposition}

\begin{proof}
A complete order embedding is by definition isometric, so we only need to show that an isometric c.p.\ order zero map $\theta\colon A\to B$ is a complete order embedding; by Remark \ref{rmk: folklore 1}(ii) this reduces to showing that $\theta$ is completely isometric. 
Without loss of generality, we assume $B=\Cstar(\theta(A))$. Let $\iota\colon A\longrightarrow CA$ denote the canonical embedding $a\mapsto \id_{(0,1])}\otimes a$ of $A$ into its cone. Then by \cite[Corollary 4.1]{WZ09}, $\theta$ induces a surjective $^*$-homomorphism $\hat{\theta}:CA\longrightarrow B$ so that $\hat{\theta} \circ\iota =\theta$. 

Set $\pi\colon B\longrightarrow B/\hat{\theta}(SA)$ and $\hat{\pi}\coloneqq \pi\circ\hat{\theta}:CA\longrightarrow B/\hat{\theta}(SA)$ where $SA\coloneqq C_0((0,1))\otimes A \cong C_0((0,1),A)$ denotes the suspension of $A$. 
Then there exists a surjective $^*$-homomorphism $\varphi\colon A\longrightarrow B/\hat{\theta}(SA)$ so that $\hat{\pi}=\varphi\circ \text{ev}_1$. 
Altogether, we have the commutative diagram
\[
\begin{tikzcd}
A\arrow[swap]{d}{\iota}\arrow{r}{\theta} & B \arrow[two heads]{d}{\pi} \\
CA\arrow[two heads]{ur}{\hat{\theta}}\arrow[two heads]{r}{\hat{\pi}} \arrow[two heads, swap]{d}{\text{ev}_1} & B/\hat{\theta}(SA)\\
A\arrow[two heads]{ur}{\varphi}[swap]{^*\text{-hom}} &
\end{tikzcd}.
\]
We claim that $\varphi$ is injective. 
If so, then we have $\varphi^{-1}\circ \pi \circ \theta =\id_A$, and hence $(\varphi^{-1})^{(r)}\circ \pi^{(r)} \circ \theta^{(r)} =\id_A^{(r)}$ for each $r\geq 1$. Since $\varphi^{-1}$ and $\pi$ are in this case completely contractive, it will then follow that $\theta$ must be completely isometric. 

It suffices to check injectivity on $A_+$, and so we fix $a\in A_+$ with $\|a\|=1$. Let $\rho\in S(B)$ be a pure state on $B$ with $\|\theta(a)\|=\rho(\theta(a))$. Then $\rho\circ\hat{\theta}$ is a pure state on $CA$ and hence its restriction to the algebraic tensor product $C_0((0,1])\odot A$ is of the form $\text{ev}_t\otimes \rho$ for some $t\in (0,1]$ and some (pure) state $\rho_A$ on $A$. Since 
\[1=\|a\|=\|\theta(a)\|=\rho\circ\hat{\theta}(\id_{(0,1]}\otimes a)=t\rho_A(a)\leq \rho_A(a)\leq \|a\|,\]
we must have $t=1$. In particular, $SA\subset \text{ker}(\rho\circ\hat{\theta})$, and hence $\hat{\theta}(SA)\subset \text{ker}(\rho)$. Since $\rho(\theta(a))=1$, we conclude that $\theta(a)\notin \hat{\theta}(SA)$, and so $\varphi(a)\neq 0$. 
\end{proof}

%=====================================

\section{CPC*-systems and their limits}\label{sect: cpc systems}
\noindent Recall from Remark \ref{cor: order zero for unital}(ii) that for $\Cstar$-algebras $A$ and $B$ with $A$ unital a c.p.c.\ map $\theta\colon A\longrightarrow B$ is order zero precisely when  $\theta(a)\theta(b)=\theta(1_A)\theta(ab)$ for all $a,b\in A$. The following definition gives an asymptotic version of this characterization.

\begin{definition}\label{def: limit}
Given a sequence of $\Cstar$-algebras $(F_n)_n$ we let 
\[F_\infty\coloneqq \textstyle{\prod F_n/\bigoplus F_n}\] be the quotient $\Cstar$-algebra of norm bounded sequences modulo null  sequences, denoting elements of the quotient in the form $[(x_n)_n]$.

If in addition we have c.p.c.\ connecting maps $\rho_{n+1,n}\colon F_n\longrightarrow F_{n+1}$ for all $n\geq 0$, we set $\rho_{m,n}\coloneqq \rho_{m,m-1}\circ\hdots\circ\rho_{n+1,n}$ and $\rho_{n,n}=\id_{F_n}$ for all $m>n\geq 0$, and define induced maps $\rho_n\colon F_n\longrightarrow F_\infty$ by  
\begin{align*}
       \rho_n(x)=[(\rho_{m,n}(x))_{m>n}]
\end{align*}
for all $x\in F_n$ and define the inductive limit of the system $(F_n, \rho_{n+1,n})_n$ to be the closed self-adjoint subspace 
\begin{align*}
    \Lim &\coloneqq \overline{\textstyle{\bigcup_n}\rho_n(F_n)}\subset F_\infty. 
\end{align*}
\end{definition}

In general such an inductive limit exists in the category of operator spaces; in order to equip it with the structure of a $\Cstar$-algebra one needs additional hypotheses.

\begin{definition}\label{def: cpcstar}
Let $(F_n)_n$ be a sequence of finite-dimensional $\Cstar$-algebras together with c.p.c.\ maps $\rho_{n+1,n}\colon F_n\longrightarrow F_{n+1}$ for all $n\geq 0$. We call $(F_n,\rho_{n+1,n})_n$ a \emph{$\CPCstar$-system} if for any $k\geq 0$, $x,y\in F_k$, and $\varepsilon>0$, there exists $M>k$ so that for all $m>n,l>M$, 
\begin{align}\label{eq: asym perp}
    \|\rho_{m,l}(1_{F_l})\rho_{m,n}(\rho_{n,k}(x)\rho_{n,k}(y))-\rho_{m,k}(x)\rho_{m,k}(y)\|<\varepsilon.
\end{align} 
In this situation, we call the limit $\Lim$ a $\CPCstar$-limit.
\end{definition}

The aim of this section is to prove that a $\CPCstar$-limit can be turned, in a very robust sense, into a $\Cstar$-algebra (\cref{thm: main 1}), denoted by $\Cstar_{\sbt}\big(\Lim \big)$, which is nuclear (\cref{thm: nuclear}), and such that the map $\Theta\coloneqq \id_{\Lim }\colon \Cstar_{\sbt}\big(\Lim \big)\longrightarrow F_\infty$ is an order zero complete order embedding (\cref{thm: order zero coi}). 
We begin by accumulating some tools and observations.

Systems as in \cref{def: limit} involve c.p.c.\ maps, hence they are clearly compatible with taking matrix amplifications. It is slightly less obvious that the process of lifting individual elements of the inductive limit is also compatible with the system in the following sense.

\begin{lemma}\label{lem: limit of a system}
Let $(F_n,\rho_{n+1,n})_n$ be a system as in \emph{\cref{def: limit}}. 
        For any $r\geq 1$ and $\bar{x}\in \M_r\big( \Lim \big)$, if $(x_n)_n\in \prod \M_r(F_n)$ is some lift of $\bar{x}$, then $\lim_n\|\bar{x}-\rho_n^{(r)}(x_n)\|=0$. 
\end{lemma}

\begin{proof}
Consider $r=1$ and fix $\bar{x}\in  \Lim $, a lift $(x_n)_n\in \prod F_n$ of $\bar{x}$, and $\varepsilon>0$. 
Let $y_{k_j}\in F_{k_j}$ so that $\bar{x}=\lim_j\rho_{k_j}(y_{k_j})$. Since the maps are coherent, we can assume for simplicity that the $k_j$ are increasing and set $y_n\coloneqq  \rho_{n,k_j}(y_{k_j})$ for $k_j<n<k_{j+1}$. 
Then we have $\lim_n\rho_n(y_n)= \bar{x}$. 
Now fix $k>0$ so that $\|\rho_k(y_k)-\bar{x}\|<\varepsilon/4$. 
Then we may choose $M>k$ so that $\|\rho_{n,k}(y_k)-x_n\|<\varepsilon/2$ for all $n>M$. Then for all $m>n>M$, 
\begin{align*}
    \|\rho_{m,n}(x_n)-x_m\|&\leq \|\rho_{m,n}(x_n-\rho_{n,k}(y_k))\| + \|\rho_{m,k}(y_k)-x_m\|\\&\leq \|x_n-\rho_{n,k}(y_k)\| + \|\rho_{m,k}(y_k)-x_m\|\\
    &<\varepsilon.
\end{align*}
It follows that $\lim_n \rho_n(x_n)=[(x_n)_n]=\bar{x}$. 

Note that the argument for $r=1$ uses only that the maps $\rho_{m,n}$ are coherent and c.p.c., which is also true for their matrix amplifications, and so  the exact same proof goes through for $r\geq 1$. 
\end{proof}

A key tool for us will be a distinguished matrix order unit\footnote{This terminology is justified by \cref{lem: uniform order unit}\ref{lemit: uniform order unit}.} for the limit of a $\CPCstar$-system $(F_n,\rho_{n+1,n})_n$, which we define to be 
\begin{align}
    e&\coloneqq [(\rho_{n+1,n}(1_{F_n}))_n]\in F_\infty. \label{def: order unit}
\end{align}
Note that we do not in general assume that $e$ is an element of $\Lim $. That latter situation would correspond to $\Cstar_{\sbt}\big(\Lim \big)$ being unital (with $e=\lim_n\rho_n(1_{F_n})$), in which case \eqref{eq: asym perp} can be slightly simplified by eliminating the variable $l$.

We also set 
\begin{align*}
    e^{(r)}\coloneqq [(\rho^{(r)}_{n+1,n}(1_{\M_r(F_n)})_n]\in \M_r(F_\infty),\ \text{ for }\ r\geq 1.
\end{align*}
\begin{remark}\label{lemit: asym perp}
Note that \eqref{eq: asym perp} implies that for any $k\geq 0$ and $x,y\in F_k$, we have 
   \[\rho_k(x)\rho_k(y)=\lim_n e \rho_n\big(\rho_{n,k}(x)\rho_{n,k}(y)\big).\] 
\end{remark}
The following lemma will allow us to interpret $e$ as a commuting nondegenerate order unit for $\Lim$.

\begin{lemma}\label{lem: uniform order unit}
Let $(F_n,\rho_{n+1,n})_n$ be a  $\CPCstar$-system, and let $e$ be as defined in \eqref{def: order unit}.
\begin{enumerate}[label=\textnormal{(\roman*)}]
\item For each $k\geq 0$, $x\in F_k$, and $\varepsilon>0$, there exists $M>k$ so that for all $m>n>M$, 
    \[\|\rho_{m,n}(1_{F_n})\rho_{m,k}(x)-\rho_{m,k}(x)\rho_{m,n}(1_{F_n})\|<\varepsilon.\] 
    In particular, $e\in \big(\Lim \big)'$. \label{lemit: asym comm} 
        \item For any $r\geq 1$ and any self-adjoint $\bar{x}\in \M_r(\Lim )$, we have $\|\bar{x}\|e^{(r)}\geq \bar{x}$ in $\M_r(F_\infty)$. \label{lemit: uniform order unit}
    \item For any $\bar{x}\in \Lim $ and any $j\geq 1$, we have 
$\|e^j\bar{x}\|=\|\bar{x}\|$. In particular, for any $\bar{x},\bar{y}\in \Lim $ and $j\geq 1$, if $e^j\bar{x}=e^j\bar{y}$, then $\bar{x}=\bar{y}$. \label{lemit: isometric}
\end{enumerate} 
\end{lemma}
\begin{proof}
For \ref{lemit: asym comm}, we fix $k\geq 0$, $x\in F_k$, and $\varepsilon>0$; without loss of generality, it suffices to assume $x\in (F_k)_+^1$. 
It follows from a standard functional calculus argument using polynomial approximations of $t\mapsto t^{1/2}$ on $[0,1]$
that we are finished when we find for any $\eta>0$ an $M>k$ so that for all $m>n>M$, 
\begin{align}\label{eq: *}
    \|\rho_{m,n}(1_{F_n})\rho_{m,k}(x)^2-\rho_{m,k}(x)^2\rho_{m,n}(1_{F_n})\|<\eta.
\end{align}
To that end, we utilize \eqref{eq: asym perp} to choose $M>k$ so that for all $m>n>M$
\begin{align}\label{eq: 6eta}
\|\rho_{m,n}(1_{F_n})\rho_{m,n}(\rho_{n,k}(x)^2)-\rho_{m,k}(x)^2\|<\eta/6.\end{align}
Then we have

  \vspace{-.25 cm}

\begin{align*}
    &\|\rho_{m,n}(1_{F_n})\rho_{m,k}(x)^2-\rho_{m,k}(x)^2\rho_{m,n}(1_{F_n})\|\\
    &\overset{\phantom{\eqref{eq: 6eta}}}{\leq}   \|\rho_{m,n}(1_{F_n})\rho_{m,k}(x)^2-\rho_{m,n}(1_{F_n})^2\rho_{m,n}(\rho_{n,k}(x)^2)\| \\
    & \phantom{\overset{\eqref{eq: 6eta}}{<}} +    \|\rho_{m,n}(1_{F_n})^2\rho_{m,n}(\rho_{n,k}(x)^2)-\rho_{m,n}(1_{F_n})\rho_{m,n}(\rho_{n,k}(x)^2)\rho_{m,n}(1_{F_n})\|\\
   & \phantom{\overset{\eqref{eq: 6eta}}{<}} +    \|\rho_{m,n}(1_{F_n})\rho_{m,n}(\rho_{n,k}(x)^2)\rho_{m,n}(1_{F_n})-\rho_{m,n}(\rho_{n,k}(x)^2)\rho_{m,n}(1_{F_n})^2\|\\
   &\phantom{\overset{\eqref{eq: 6eta}}{<}} +\|\rho_{m,n}(\rho_{n,k}(x)^2)\rho_{m,n}(1_{F_n})^2-\rho_{m,k}(x)^2\rho_{m,n}(1_{F_n})\|\\
    &\overset{\phantom{\eqref{eq: 6eta}}}{\leq}  \|\rho_{m,n}(1_{F_n})\|\|\rho_{m,k}(x)^2-\rho_{m,n}(1_{F_n})\rho_{m,n}(\rho_{n,k}(x)^2)\| \\
    & \phantom{\overset{\eqref{eq: 6eta}}{<}} +    \|\rho_{m,n}(1_{F_n})\|\|\rho_{m,n}(1_{F_n})\rho_{m,n}(\rho_{n,k}(x)^2)-\rho_{m,n}(\rho_{n,k}(x)^2)\rho_{m,n}(1_{F_n})\|\\
   & \phantom{\overset{\eqref{eq: 6eta}}{<}} +    \|\rho_{m,n}(1_{F_n})\rho_{m,n}(\rho_{n,k}(x)^2)-\rho_{m,n}(\rho_{n,k}(x)^2)\rho_{m,n}(1_{F_n})\|\|\rho_{m,n}(1_{F_n})\|\\
   &\phantom{\overset{\eqref{eq: 6eta}}{<}} +\|\big(\rho_{m,n}(\rho_{n,k}(x)^2)\rho_{m,n}(1_{F_n})-\rho_{m,k}(x)^2\big)^*\|\|\rho_{m,n}(1_{F_n})\|\\
    &\leq 2\|\rho_{m,k}(x)^2-\rho_{m,n}(1_{F_n})\rho_{m,n}(\rho_{n,k}(x)^2)\|\\
    &\phantom{\overset{\eqref{eq: 6eta}}{<}}  +    2\|\rho_{m,n}(1_{F_n})\rho_{m,n}(\rho_{n,k}(x)^2)-\rho_{m,n}(\rho_{n,k}(x)^2)\rho_{m,n}(1_{F_n})\| \\
   &\overset{\eqref{eq: 6eta}}{<}  \eta/3+  2\|\rho_{m,n}(1_{F_n})\rho_{m,n}(\rho_{n,k}(x)^2) - \rho_{m,k}(x)^2\|\\
    &\phantom{\overset{\eqref{eq: 6eta}}{<}}  + 2\|\rho_{m,k}(x)^2-\rho_{m,n}(\rho_{n,k}(x)^2)\rho_{m,n}(1_{F_n})\|\\
    &\overset{\phantom{\eqref{eq: 6eta}}}{=}  \eta/3 +  2\|\rho_{m,n}(1_{F_n})\rho_{m,n}(\rho_{n,k}(x)^2) - \rho_{m,k}(x)^2\|\\
   &\phantom{\overset{\eqref{eq: 6eta}}{<}}  + 2\|\big(\rho_{m,k}(x)^2-\rho_{m,n}(\rho_{n,k}(x)^2)\rho_{m,n}(1_{F_n})\big)^*\|\\
     &\overset{\phantom{\eqref{eq: 6eta}}}{=} \eta/3+4\|\rho_{m,n}(1_{F_n})\rho_{m,n}(\rho_{n,k}(x)^2) - \rho_{m,k}(x)^2\|\\
    &\overset{\eqref{eq: 6eta}}{<}\eta, 
\end{align*}
which establishes \eqref{eq: *} and hence \ref{lemit: asym comm}.

% To that end, we utilize \eqref{eq: asym perp} to choose $M>k$ so that for all $m>n>M$
% \[\|\rho_{m,n}(1_{F_n})\rho_{m,n}(\rho_{n,k}(x)^2)-\rho_{m,k}(x)^2\|<\eta/2.\]
% Then we have 

% \begin{align*}
%     &\|\rho_{m,n}(1_{F_n})\rho_{m,k}(x)^2-\rho_{m,k}(x)^2\rho_{m,n}(1_{F_n})\|\\
%     &\leq  \|\rho_{m,n}(1_{F_n})\rho_{m,k}(x)^2-\rho_{m,n}(1_{F_n})\rho_{m,n}(\rho_{n,k}(x)^2)\rho_{m,n}(1_{F_n})\| \\
%    & \phantom{{}\leq{}} +    \|\rho_{m,n}(1_{F_n})\rho_{m,n}(\rho_{n,k}(x)^2)\rho_{m,n}(1_{F_n})-\rho_{m,k}(x)^2\rho_{m,n}(1_{F_n})\|\\
%    &< \|\big(\rho_{m,k}(x)^2-\rho_{m,n}(1_{F_n})\rho_{m,n}(\rho_{n,k}(x)^2)\big)^*\|\\
%    & \phantom{{}\leq{}} +  \|\rho_{m,n}(1_{F_n})\rho_{m,n}(\rho_{n,k}(x)^2)-\rho_{m,k}(x)^2\|\\
%    &<\eta, 
% \end{align*}
% which establishes \eqref{eq: *} and hence \ref{lemit: asym comm}.

For the order unit condition \ref{lemit: uniform order unit}, we first consider $r=1$. For any self-adjoint $\bar{x}\in \Lim $ and any self-adjoint lift $(x_n)_n\in \prod F_n$ of $\bar{x}$, \cref{lem: limit of a system} 
tells us $\bar{x}=\lim_n\rho_n(x_n)$. 
Hence, it suffices to prove the claim for $\rho_k(x)$ for any fixed $k\geq 0$ and any self-adjoint $x\in F_k$. For each $n>k$, we have $\|\rho_{n,k}(x)\|1_{F_n}\geq \rho_{n,k}(x)$, and so $\|\rho_{n,k}(x)\|\rho_n(1_{F_n})\geq \rho_n(\rho_{n,k}(x))=\rho_k(x)$. Since all the maps are c.p.c., it follows that $e\geq \rho_n(1_{F_n})$ for all $n\geq 0$, and hence $\|\rho_{n,k}(x)\| e \geq \rho_k(x)$ for all $n>k$. Moreover $(\|\rho_{n,k}(x)\|)_{n>k}$ is bounded and non-increasing, which means
$\|\rho_k(x)\|=\limsup_n\|\rho_{n,k}(x)\|=\lim_n \|\rho_{n,k}(x)\|$, and so 
\[\|\rho_k(x)\|e=\lim_n \|\rho_{n,k}(x)\| e \geq \rho_k(x).\]
Since the argument for $r=1$ uses only \cref{lem: limit of a system} and that the maps are c.p.c., we may repeat the same argument for $r\geq 1$.

For the nondegeneracy condition \ref{lemit: isometric}, we assume $\Lim \neq \{0\}$ (otherwise the claim is trivial). We first consider the case where $\bar{x}\in \Lim $ is self-adjoint, and we assume moreover that $\|\bar{x}\|=1$. By possibly replacing $\bar{x}$ with $-\bar{x}$, we may assume $1\in \sigma(\bar{x})$, the spectrum of $\bar{x}$. From \ref{lemit: asym comm}, we know $e$ and $\bar{x}$ commute, and so  
we may identify $\Cstar(e,\bar{x})\subset F_\infty$ with $C_0(\Omega)$ for some locally compact Hausdorff space $\Omega\subset \sigma(e)\times \sigma(\bar{x})\subset [0,1]\times [-1,1]$. Then \ref{lemit: uniform order unit} implies $s\geq t$ for all $(s,t)\in \Omega$.  
Since we assumed $1\in \sigma(\bar{x})$, it follows that  $(1,1)\in \Omega$, and so for any $m\geq 1$,
\[1=\|\bar{x}\|\geq \|e^m\bar{x}\|=\sup_{(s,t)\in \Omega} |s^m t|\geq 1.\]

Now, we handle the general case of any $\bar{x}\in \Lim $. Because $\bigcup_n \rho_n(F_n)$ is dense in $\Lim $, it suffices to check that \ref{lemit: isometric} holds for any fixed $\bar{x}=\rho_k(x)$ with $k\geq 0$ and $x\in F_k$. 
From \cref{lemit: asym perp}, we know that 
\begin{align}
    \rho_k(x)^*\rho_k(x)=\lim_n e \rho_n\big(\rho_{n,k}(x)^*\rho_{n,k}(x)\big).\label{eq: **}
\end{align}
Since $\rho_m\big(\rho_{m,k}(x)^*\rho_{m,k}(x)\big)-\rho_n\big(\rho_{n,k}(x)^*\rho_{n,k}(x)\big)\in \Lim$ is self-adjoint for all $m>n>k$, we already know that
\begin{align*}
   &\|e \rho_m\big(\rho_{m,k}(x)^*\rho_{m,k}(x)\big) - e \rho_n\big(\rho_{n,k}(x)^*\rho_{n,k}(x)\big)\|\\
   &= \|\rho_m\big(\rho_{m,k}(x)^*\rho_{m,k}(x)\big) - \rho_n\big(\rho_{n,k}(x)^*\rho_{n,k}(x)\big)\|,
\end{align*}
which together with \eqref{eq: **} implies that  $\big(\rho_n\big(\rho_{n,k}(x)^*\rho_{n,k}(x)\big)\big)_n$ is Cauchy and hence converges to some self-adjoint $\bar{z}\in \Lim $ with $e \bar{z}=\rho_k(x)^*\rho_k(x)$. From this and the self-adjoint case we have
\begin{align*}
    \|\rho_k(x)\|^2=\|\rho_k(x)^*\rho_k(x)\|=\|e\bar{z}\|=\|\bar{z}\|.
\end{align*}
Since $\bar{z}$ is self-adjoint \ref{lemit: uniform order unit} tells us 
$  \|\rho_k(x)\|^2 e =\|\bar{z}\|e\geq \bar{z}$. By \ref{lemit: asym comm}, $e$ commutes with $\bar{z}$ and moreover 
with $|\rho_k(x)|$. 
It follows that 
\[\|\rho_k(x)\|^2 e^2 \geq e\bar{z} = \rho_k(x)^*\rho_k(x),\]
which implies
\[\|\rho_k(x)\| e \geq |\rho_k(x)|.\]
With this, the proof of the self-adjoint case also shows that $\|e^m |\rho_k(x)|\|\\ =\| |\rho_k(x)| \|$ for any $m\geq 1$, and so it follows that for any $m\geq 1$ we have 
\begin{align*}\|e^m\rho_k(x)\|^2&=
\|e^{2m}|\rho_k(x)|^2\|\\
&=\|(e^m|\rho_k(x)|)^2\|\\
&=\|e^m|\rho_k(x)|\|^2\\
&=\||\rho_k(x)|\|^2\\
&=\|\rho_k(x)\|^2. \qedhere\end{align*}
\end{proof}

With these preliminary observations covered, we are now ready to prove our first structural result about $\CPCstar$-limits. 

\begin{proposition}\label{thm: main 1}
Let $\Lim $ be the limit of a  $\CPCstar$-system\\ $(F_n,\rho_{n+1,n})_n$. Then there exists an associative bilinear map\\ $\sbt \colon \Lim  \times \Lim  \longrightarrow \Lim $ satisfying
\begin{align}
    \rho_k(x)\sbt \rho_k(y) & = \lim_n \rho_n(\rho_{n,k}(x) \rho_{n,k}(y)) \label{eq: explicit mult}
    \end{align}
    for all  $k\geq 0$ and $x,y\in F_k$ and 
    \begin{align}
    e(\bar{x}\sbt \bar{y})&=\bar{x} \bar{y} \label{eq: mult id}
\end{align}
for all $\bar{x},\bar{y}\in \Lim$, 
so that with respect to the multiplication $\sbt$,  $\Lim $ is a $\Cstar$-algebra with norm $\|\cdot\|_{F_\infty}$, denoted $\Cstar_{\sbt}\big(\Lim \big)$.\footnote{We are simultaneously viewing $\Lim $ as a subspace of $F_\infty$ and as the space $\Cstar_{\sbt}\big(\Lim \big)$. For the sake of clarity, multiplication in $F_\infty$ will always be denoted by the usual concatenation (e.g., $\bar{x}\bar{y}$), and multiplication in $\Cstar_{\sbt}\big(\Lim \big)$ will always be denoted with the $\sbt$ (e.g., $\bar{x}\sbt \bar{y}$). So, for $e(\bar{x}\sbt \bar{y})$, we are taking the $F_\infty$-product of $e$ and $\bar{x}\sbt\bar{y}$. }
\end{proposition}

\begin{proof}
We begin by defining $\sbt\colon\bigcup_n \rho_n(F_n) \times \bigcup_n \rho_n(F_n) \longrightarrow \Lim $. 
For elements $\bar{x},\bar{y}$ in the nested union $\bigcup_n \rho_n(F_n)$ 
we may choose $k\geq 0$ and $x,y\in F_k$ so that $\bar{x}=\rho_k(x)$ and $\bar{y}=\rho_k(y)$ and define $\bar{x}\sbt \bar{y}$ by 
\begin{align}
     \bar{x}\sbt \bar{y} & \coloneqq  \lim_n \rho_n(\rho_{n,k}(x) \rho_{n,k}(y)). 
\end{align}
That the given limit exists follows immediately from \cref{lemit: asym perp} and \cref{lem: uniform order unit}\ref{lemit: isometric}.  
Moreover, this definition is independent of the particular choice of $k$ and $x,y\in F_k$. Indeed, for any $(x_n)_n, (y_n)_n\in \prod F_n$ with $[(x_n)_n]=\rho_k(x)=\bar{x}$ and $[(y_n)_n]=\rho_k(y)=\bar{y}$ we have $[(\rho_{n,k}(x)\rho_{n,k}(y))_n]=[(x_ny_n)_n]$. Then for any $\varepsilon>0$ and $M>k$ with $\sup_{n>M}\|\rho_{n,k}(x)\rho_{n,k}(y)-x_ny_n\|<\varepsilon$, we have $\|\rho_n(\rho_{n,k}(x)\rho_{n,k}(y))-\rho_n(x_ny_n)\|<\varepsilon$ for all $n>M$, and so the sequence $\rho_n(x_ny_n)\in \Lim $ also converges to  $\lim_n \rho_n(\rho_{n,k}(x) \rho_{n,k}(y))$. 

Bilinearity of this map follows from that of the usual product in the $F_n$'s. For each $k\geq 0$ and $x,y\in F_k$ we have from \cref{lemit: asym perp} that
\begin{align}\label{lemit: mult on dense}
\rho_k(x)\rho_k(y) =  e(\rho_k(x)\sbt \rho_k(y)),
\end{align}
which by the nondegeneracy condition in \cref{lem: uniform order unit}\ref{lemit: isometric} implies
\begin{align}\label{eq: C norm}
    \|\rho_k(x)\sbt \rho_k(y)\| =  \|e(\rho_k(x)\sbt \rho_k(y))\| =  \|\rho_k(x)\rho_k(y)\|.
\end{align}
For $\bar{x},\bar{y}\in \Lim$ and $(x_n)_n, (y_n)_n\in \prod F_n$ with $\bar{x}=[(x_n)_n]$ and $\bar{y}=[(y_n)_n]$, it follows from \eqref{lemit: mult on dense} and \cref{lem: limit of a system} that
\[\lim_n e\big(\rho_n(x_n)\sbt\rho_n(y_n)\big)=\lim_n \rho_n(x_n)\rho_n(y_n)=xy.\] Moreover, by 
\cref{lem: uniform order unit}\ref{lemit: isometric} we have for all $m>n\geq 0$
\begin{align*}
    &\|e\big(\rho_m(x_m)\sbt\rho_m(y_m)-\rho_n(x_n)\sbt\rho_n(y_n)\big)\|\\ &=\|\rho_m(x_m)\sbt\rho_m(y_m)-\rho_n(x_n)\sbt\rho_n(y_n)\|.
\end{align*}
Hence $\lim_n \rho_n(x_n) \sbt \rho_n(y_n)$ exists, and we may extend the map to\\ $\sbt\colon  
\Lim 
\times \Lim  
\longrightarrow \Lim $ by defining 
\[\bar{x} \sbt \bar{y} \coloneqq  \lim_n \rho_n(x_n) \sbt \rho_n(y_n)\]
for each $(\bar{x},\bar{y})\in \Lim 
\times \Lim $, where $(x_n)_n, (y_n)_n\in \prod F_n$ are any lifts of $\bar{x}$ and $\bar{y}$. Note that we still have 
\[e(\bar{x} \sbt \bar{y}) = e\big(\lim_n \rho_n(x_n) \sbt \rho_n(y_n)\big) = \lim_n \rho_n(x_n) \rho_n(y_n) = \bar{x}\bar{y},\]
and so it follows again from nondegeneracy (\cref{lem: uniform order unit}\ref{lemit: isometric}) that $\bar{x} \sbt \bar{y}$ is the unique element $\bar{z}\in \Lim $ so that $e \bar{z} = \bar{x}\bar{y}$. In summary, our map $\sbt\colon \Lim 
\times \Lim  
\longrightarrow \Lim $ is well-defined, bilinear and satisfies \eqref{eq: mult id} and \eqref{eq: explicit mult}. To see that it is associative, fix $\bar{x}, \bar{y}, \bar{z} \in \Lim $. Then we have 
\begin{align*}
    e^2((\bar{x} \sbt \bar{y}) \sbt \bar{z}) &= e ((\bar{x} \sbt \bar{y}) \bar{z})\\ &= ((\bar{x}\bar{y}) \bar{z})\\ &= (\bar{x} (\bar{y} \bar{z}))\\ &= \bar{x} e(\bar{y}\sbt\bar{z})\\
    &=e (\bar{x} (\bar{y} \sbt \bar{z}))\\ &= e^2(\bar{x} \sbt (\bar{y} \sbt \bar{z})).
\end{align*}
It follows from nondegeneracy 
that 
\begin{align*}
    (\bar{x} \sbt \bar{y}) \sbt \bar{z} = \bar{x} \sbt (\bar{y} \sbt \bar{z}).
\end{align*}

So, $(\Lim , \sbt)$ is an algebra. To see that it is a $^*$-algebra with respect to the involution it inherits from $F_\infty$, we set $\bar{x}, \bar{y} \in \Lim $, and observe that
\begin{align*}
    e(\bar{x} \sbt \bar{y})^*=(\bar{x} \sbt \bar{y})^*e=(e(\bar{x} \sbt \bar{y}))^*=(\bar{x}\bar{y})^* = \bar{y}^*\bar{x}^* = e(\bar{y}^*\sbt \bar{x}^*).
\end{align*}
Again it follows from nondegeneracy 
that $(\bar{x}\sbt\bar{y})^*=\bar{x}\sbt\bar{y}$, so indeed $(\Lim , \sbt)$ is a $^*$-algebra. 

Since $(\Lim , \sbt)$ is already a Banach space with respect to the given norm on $F_\infty$, it remains to confirm that it is in fact a Banach $^*$-algebra and that the norm still satisfies the $\Cstar$-identity. Both of these follow immediately from the fact that  
\begin{align}\label{eq: same norm on prod}
    \|\bar{x}\sbt \bar{y}\| = \|e(\bar{x}\sbt \bar{y})\| = \|\bar{x} \bar{y}\| 
\end{align} 
for any 
$\bar{x}, \bar{y} \in \Lim $.
\end{proof}

\begin{proposition}\label{thm: order zero coi}
Let $(F_n,\rho_{n+1,n})_n$ be a $\CPCstar$-system. 
Then \[\Theta \coloneqq  \id_{\Lim } \colon  \Cstar_{\sbt}\big(\Lim \big) \longrightarrow F_\infty\] is an order zero complete order embedding. 
\end{proposition}

\begin{proof}
Clearly the map is isometric, $^*$-preserving, and linear since the given $^*$-linear structures and norms on $\Cstar_{\sbt}\big(\Lim \big)$ and $\Lim  \subset F_\infty$ are the same. Moreover, from \eqref{eq: same norm on prod} we can conclude that $\Theta$ is order zero. By \cref{prop: isom perp is coe}, it remains  
only to show that it is completely positive, i.e., for any $r\geq 1$ and $\bar{x} \in \M_r(\Cstar_{\sbt}\big(\Lim \big))$, the element $\bar{x}^*\sbt \bar{x}\in \M_r(\Lim )$ is in $\M_r(F_\infty)_+$. (We still use $\sbt$ to denote the multiplication in $\M_r(\Cstar_{\sbt}\big(\Lim \big))$.)

Fix $r\geq 1$, and note that by applying  \eqref{eq: mult id} coordinate-wise we have for any $\bar{x} \in \M_r(\Cstar_{\sbt}\big(\Lim \big))$ that 
\[e^{(r)}(\bar{x}^*\sbt \bar{x}) = \bar{x}^*\bar{x}\geq 0.\]
So it suffices to show that $e^{(r)}\bar{y}\geq 0$ implies $\bar{y}\geq 0$ for any $\bar{y}\in \M_r(\Lim )\subset \M_r(F_\infty)$. 

To that end, fix $\bar{y}\in \M_r(\Lim )\subset \M_r(F_\infty)$ and assume $e^{(r)}\bar{y}\geq 0$. If $y=0$, there is nothing to show, so we assume 
$\|\bar{y}\|=1$. 
Using \cref{lem: uniform order unit}\ref{lemit: isometric}, we also get a nondegeneracy condition for $e^{(r)}$, which says that $e^{(r)}\bar{z}=0$ implies $\bar{z}=0$ for any 
$\bar{z}\in \M_r(\Lim )\subset \M_r(F_\infty)$.
Then since 
\[e^{(r)}\bar{y}= (e^{(r)}\bar{y})^*= e^{(r)}\bar{y}^*,\]
we conclude that $\bar{y}= \bar{y}^*$. 

From \cref{lem: uniform order unit}\ref{lemit: uniform order unit}, we know $e^{(r)}- \bar{y}$ and $e^{(r)}+ \bar{y}\in \M_r(F_\infty)_+$.  
Since $e^{(r)}$ and $\bar{y}$ commute (by \cref{lem: uniform order unit}\ref{lemit: asym comm}), this implies that $e^{(r)} \geq |\bar{y}|$ and that we can identify $\Cstar(e^{(r)},\bar{y})\subset \M_r(F_\infty)$ with $C_0(\Omega)$ for some locally compact $\Omega\subset \sigma(e^{(r)})\times \sigma(\bar{y})$, 
where $e^{(r)}$ corresponds to the map $(s,t)\mapsto s$ and $\bar{y}$ corresponds to the map $(s,t)\mapsto t$. 
Since $e^{(r)} \geq |\bar{y}|$, we have $s \geq |t|$ for all $(s,t)\in \Omega$. Now, suppose there exists $t_0\in \sigma(\bar{y})\cap (-\infty,0)$, and let $s_0\in \sigma(e^{(r)})$ so that $(s_0,t_0)\in \Omega$. Then $s_0 \geq |t_0|>0$ implies $s_0>0$, whence $s_0t_0<0$, contradicting $e^{(r)}\bar{y}\geq 0$. Hence $\bar{y}\geq 0$, and our proof is complete. 
\end{proof}

\begin{corollary}\label{prop: CPC*-algebras'}
Let $(F_n,\rho_{n+1,n})_n$ be a $\CPCstar$-system, and let $\Theta$ be as in \emph{\cref{thm: order zero coi}}.
Then for any $\Cstar$-algebra $A$ and complete order embedding $\psi\colon A\longrightarrow F_\infty$ with $\psi(A)= \Lim $, 
the map $\Theta^{-1}\circ\psi\colon A\longrightarrow \Cstar_{\sbt}\big(\Lim \big)$ is a $^*$-isomorphim. 
\end{corollary}

\begin{proof}
With \cref{thm: order zero coi} we see that $\Theta^{-1}\circ\psi\colon A\longrightarrow \Cstar_{\sbt}\big(\Lim \big)$ is a surjective complete order isomorphism. 
The claim then follows from Remark \ref{rmk: folklore 1}(i).
\end{proof}

\begin{remark}

By definition, $\Lim$ and $\Cstar_{\sbt}\big(\Lim \big)$ are identical as involutive Banach spaces. With the previous results at hand, we may describe the $\Cstar$-algebra structure in two alternative ways: 
Let $C$ denote the $\Cstar$-algebra generated by $\Lim $ in $F_\infty$ and $\pi_\Theta\colon\Cstar_{\sbt}\big(\Lim \big)\longrightarrow \mathcal{M}(C)$ the $^*$-homomorphism guaranteed by Theorem \ref{thm: structure thm}. From Theorem \ref{thm: structure thm}, we see that $\ker(\pi_\Theta)=\ker(\Theta)=\{0\}$, which means $\pi_\Theta$ is injective, and $\Cstar_{\sbt}\big(\Lim \big)\cong \pi_\Theta\big(\Cstar_{\sbt}\big(\Lim \big)\big)$, so we may regard $\Cstar_{\sbt}\big(\Lim \big)$ as a $\Cstar$-algebra of multipliers of $C$. Moreover, if $e\in \Cstar_{\sbt}\big(\Lim \big)$, then $\Theta(x)=x=e\pi_\Theta(x)$ for all $x\in \Cstar_{\sbt}\big(\Lim \big)$. 

On the other hand, if $e\in \Lim $, then one can also show that $(\Lim, \{\M_r(\Lim)\cap \M_r(F_\infty)_+\}_r,e)$ forms an abstract operator system (cf.\ \cite{CE77} or \cite[Chapter 13]{Pau02}). As an operator system, this is equal to the operator system $(\Cstar_{\sbt}\big(\Lim\big), \{\M_r\big(\Cstar_{\sbt}(\Lim)\big)\}_r,e)$, which happens to also be a $\Cstar$-algebra, and so for each $r\geq 1$ 
the norm on $\M_r\big(\Cstar_{\sbt}(\Lim)\big)$ 
agrees with the matrix norm on  $\M_r\big(\Lim\big)$  
induced by the matrix order. 
Hence by \cite[Corollary 4.2]{Ham79}, we may identify $\Cstar_{\sbt}\big(\Lim \big)$ with the enveloping $\Cstar$-algebra of this operator system.
\end{remark}

\begin{remark}
Though the subspace $\Lim \subset F_\infty$ 
might not be a sub-$\Cstar$-algebra in general,  
the correspondence between $\Lim $ and $\Cstar_{\sbt}\big(\Lim \big)$ is so robust that we take the liberty to call both $\Lim $ and  $\Cstar_{\sbt}\big(\Lim \big)$ the limit of a $\CPCstar$-system, and we refer to\\  $\Cstar_{\sbt}\big(\Lim \big)$ as a \emph{$\CPCstar$-limit}.
\end{remark}

As we mentioned before, $e$ need not be in $\Lim $. Nonetheless, the sequence $(\rho_n(1_{F_n}))_n$ will still provide an approximate identity for the limit of the system.   

\begin{corollary}\label{cor: rho 1 approx id}
Let $(F_n,\rho_{n+1,n})_n$ be a $\CPCstar$-system. 
Then $(\rho_n(1_{F_n}))_n$ forms an increasing approximate identity for $\Cstar_{\sbt}\big(\Lim \big)$. 
\end{corollary}

\begin{proof}
The sequence $(\rho_n(1_{F_n}))_n$ is increasing in $F_\infty$ since the maps $\rho_{m,n}$ are all c.p.c., and so by  
\cref{thm: order zero coi}, it is also increasing in $\Cstar_{\sbt}\big(\Lim \big)$.

To see that this forms an approximate identity, it suffices to consider elements of the form $\rho_k(x)\sbt \rho_k(y)$ for some $k\geq 0$ and $x,y\in F_k$. Indeed, since $\Cstar_{\sbt}\big(\Lim \big)$ is a $\Cstar$-algebra, it follows that for any $\bar{z}\in \Lim $, there exists $\bar{x},\bar{y}\in \Lim $ so that $\bar{z}=\bar{x}\sbt\bar{y}$. Then for any lift $(x_n)_n, (y_n)_n$ of $\bar{x}$ and $\bar{y}$, respectively, \cref{lem: limit of a system}  
tells us $\rho_n(x_n)\longrightarrow \bar{x}$ and $\rho_n(y_n)\longrightarrow \bar{y}$, and so $\rho_n(x_n)\sbt\rho_n(y_n)\longrightarrow \bar{x}\sbt\bar{y}=\bar{z}$. 

Now fix $k\geq 0, x,y \in F_k$, and $\varepsilon>0$, and choose $M>k$ so that 
\begin{align*}
\|\rho_{m,n}(1_{F_n})\rho_{m,n}(\rho_{n,k}(x)\rho_{n,k}(y))-\rho_{m,k}(x)\rho_{m,k}(y)\|&<\varepsilon/2\ \text{ and}\\
\|\rho_k(x)\sbt \rho_k(y)-\rho_n(\rho_{n,k}(x)\rho_{n,k}(y))\|&<\varepsilon/2
\end{align*}
for all $m>n>M$. 
Then using \cref{lem: uniform order unit}\ref{lemit: isometric} and \eqref{eq: mult id}, we have for all $n>M$, 
\begin{align*}
    &\|\rho_n(1_{F_n})\sbt(\rho_k(x)\sbt\rho_k(y))-\rho_k(x)\sbt\rho_k(y)\|\\
    &=\|e\ \big(\rho_n(1_{F_n})\sbt(\rho_k(x)\sbt\rho_k(y)))-\rho_k(x)\sbt\rho_k(y)\big)\|\\
    &=\|\rho_n(1_{F_n})\big(\rho_k(x)\sbt \rho_k(y)\big)-\rho_k(x)\rho_k(y)\|\\
    &< \varepsilon/2 + \|\rho_n(1_{F_n})\rho_n(\rho_{n,k}(x)\rho_{n,k}(y))-\rho_k(x)\rho_k(y)\|\\
    &\leq \varepsilon. \qedhere
\end{align*}
\end{proof}

It remains to show that all $\CPCstar$-limits are nuclear. For this, we will use Ozawa and Sato's one-way-CPAP, which appeared implicitly in \cite{Ozawa2002} (via \cite{KS03}); see \cite[Theorem 5.1]{Sato2019} for the explicit statement and its  proof. 

\begin{theorem}[\cite{Sato2019, Ozawa2002}]\label{thm: Sato}
A $\Cstar$-algebra $A$ is nuclear if and only if there exists a net $\{\varphi_\lambda\colon F_\lambda\longrightarrow A\}_{\lambda\in \Lambda}$ of c.p.c.\ maps from finite-dimensional $\Cstar$-algebras $\{F_\lambda\}_{\lambda \in \Lambda}$ such that the induced c.p.c.\ map 
\begin{align*}
    \Phi=(\varphi_\lambda)_\lambda\colon \textstyle{\prod} F_\lambda/\textstyle{\bigoplus} F_\lambda &\longrightarrow \ell^\infty (\Lambda,A)/c_0(\Lambda, A),
    \end{align*}
    given by
    $\Phi([(x_\lambda)_{\lambda\in \Lambda}])=[(\varphi_\lambda(x_\lambda))_{\lambda\in \Lambda}],
$ 
satisfies
\[\iota(A^1)\subset \Phi\big((\textstyle{\prod} F_\lambda/ \textstyle{\bigoplus} F_\lambda)^1\big),\]
where $\iota\colon A \longrightarrow \ell^\infty (\Lambda,A)/c_0(\Lambda, A)$ denotes the identification of $A$ with the sub-$\Cstar$-algebra of $\ell^\infty (\Lambda,A)/c_0(\Lambda, A)$ consisting of equivalence classes of constant nets.
\end{theorem}

\begin{theorem}\label{thm: nuclear}
$\CPCstar$-limits are nuclear.
\end{theorem}

\begin{proof}
Let $(F_n,\rho_{n+1,n})_n$ be a $\CPCstar$-system with complete order embedding 
$\Theta\coloneqq  \id_{\Lim }\colon\Cstar_{\sbt}\big(\Lim \big) \longrightarrow F_\infty$ from \cref{thm: order zero coi}. For each $m\geq 0$, we define $\varphi_m\coloneqq \Theta^{-1}\circ\rho_m\colon F_m\longrightarrow \Cstar_{\sbt}\big(\Lim \big)$. It follows from \cref{thm: order zero coi} that each $\varphi_m$ is a c.p.c. map. We denote  the sequence algebra $\prod_m \Cstar_{\sbt}\big(\Lim \big)/ \bigoplus_m \Cstar_{\sbt}\big(\Lim \big)$ by $\Cstar_{\sbt}\big(\Lim \big)_\infty$, and we write $\iota\colon \Cstar_{\sbt}\big(\Lim \big)\longrightarrow \Cstar_{\sbt}\big(\Lim \big)_\infty$ 
for the 
embedding as equivalence classes of constant sequences. 
Let $\Phi\colon F_\infty\longrightarrow \Cstar_{\sbt}\big(\Lim \big)_\infty$ be the c.p.c.\ map induced by the $\varphi_m$ as in \cref{thm: Sato} with 
$\Phi([(x_m)_m])=[(\varphi_m(x_m))_m]$. 
Note that 
for $k\geq 0$ and $x\in F_k$, we have 
\begin{align*}
\Phi(\rho_k(x))&=\Phi([(\rho_{m,k}(x))_{m>k}])\\&=[(\varphi_m(\rho_{m,k}(x)))_{m>k}]\\
&=[((\Theta^{-1}\circ\rho_m)(\rho_{m,k}(x))_{m>k}]\\
&=[(\Theta^{-1}\circ\rho_k(x))_{m>k}]\\
&=\iota\circ\Theta^{-1}\circ\rho_k(x)\\
&=\iota\circ\varphi_k(x).
\end{align*}
Since these elements are dense in $\Lim$, it follows that $\Phi(\bar{x})=\iota\circ\Theta^{-1}(\bar{x})$ for each $\bar{x}\in \Lim \subset F_\infty$.  
Since $\Theta^{-1}$ is isometric, that gives us  \begin{align*}\iota\big(\Cstar_{\sbt}\big(\Lim \big)^1\big) =\iota\circ\Theta^{-1}\big(\Lim ^1\big)= \Phi\big(\Lim ^1\big)\subset \Phi(F_\infty^1).\end{align*} 
Now with \cite[Theorem 5.1]{Sato2019} (as stated above in \cref{thm: Sato}), we conclude that $\Cstar_{\sbt}\big(\Lim \big)$ is nuclear. 
\end{proof}

%================================================================================================================================

%===================================================================

%\newpage

%===================================================================================
\section{$\CPCstar$-systems from systems of c.p.c.\ approximations}\label{sect: cpap}

\noindent The goal of this section is a converse to \cref{thm: Sato}: Any separable nuclear $\Cstar$-algebra is $^*$-isomorphic to a $\CPCstar$-limit. Moreover, the $\CPCstar$-system will arise from a system of c.p.c.\ approximations. We first recall some facts  regarding systems of c.p.c.\ approximations of separable nuclear $\Cstar$-algebras.

\begin{definition}\label{def: summable}
Let $A$ be a separable nuclear $\Cstar$-algebra.

(i) 
A \emph{system of c.p.c.\ approximations} for $A$ is a sequence  $(A\xlongrightarrow{\psi_n}F_n\xlongrightarrow{\varphi_n}A)_n$, where the  $F_n$ are finite-dimensional $\Cstar$-algebras and  $A\xlongrightarrow{\psi_n}F_n\xlongrightarrow{\varphi_n}A$ are c.p.c.\ maps so that for all $a\in A$  \[\lim_n\| \varphi_n \circ \psi_n(a) - a\|=0 .\]

(ii) We call a system of c.p.c.\ approximations as above \emph{summable} if there exists a decreasing sequence $(\varepsilon_n)_n\in \ell^1(\N)_+^1$ such that for all $n>k\geq 0$
\begin{align*}
  \|\varphi_k-\varphi_n \circ \psi_n\circ\varphi_k\|<\varepsilon_n. \end{align*}
  
  (iii) We say a system of c.p.c.\ approximations as in (i) has \emph{approximately multiplicative downwards maps} if for any $a,b\in A$, we have  $\lim_n\|\psi_n(ab)-\psi_n(a)\psi_n(b)\|=0$.
  
  (iv) We say a system of c.p.c.\ approximations as in (i) has \emph{approximately order zero downwards maps} if for any $a,b\in A_+$ with $ab=0$, we have $\lim_n\|\psi_n(a)\psi_n(b)\|= 0$.

(v) Given a system of c.p.c.\ approximations as in (i), we  define for all $m>n\geq 0$
\begin{align*}
    \rho_{m,n}&\coloneqq \psi_m \circ \varphi_{m-1}\circ\hdots\circ\varphi_n\colon F_n\longrightarrow F_m\ 
    \text{ and }\\ \rho_{n,n}&\coloneqq \id_{F_n}. 
\end{align*}
thus obtaining an associated system $(F_n,\rho_{n+1,n})_n$. 
\end{definition}

\begin{remarks}\label{rmk: summable} 
(i) We will usually consider summable systems of c.p.c.\ approximations, in which case $\rho_{m,n}$ and $\psi_m\circ\varphi_n$ become arbitrarily close for $m>n$ sufficiently large. Moreover, for any $k\geq 0$ and $x\in F_k$, the sequence $(\varphi_n(\rho_{n,k}(x)))_n$ converges in $A$.

(ii) Asking for summability will not result in any loss of generality, since after passing to a subsystem we can always arrange this for any choice of $(\varepsilon_n)_n \in \ell^1(\N)^1_+$. Indeed, given a system $(A\xlongrightarrow{\psi_n}F_n\xlongrightarrow{\varphi_n}A)_n$ of c.p.c.\ approximations of a nuclear $\Cstar$-algebra $A$ and a decreasing sequence $(\varepsilon_j)\in \ell^1(\N)^1_+$, we can use compactness of the unit balls of the $F_n$ to find an increasing sequence $(n_j)_j \subset \N$ so that for all $n_k\geq 0$ 
and $j>k$,
    $$\|\varphi_{n_k}-\varphi_{n_j} \circ \psi_{n_j} \circ \varphi_{n_k}\|<\varepsilon_{n_j}.$$

(iii) For a system of c.p.c.\ approximations of a $\Cstar$-algebra $A$ as in \ref{def: summable} and  any $a\in A$, we have 
\[\lim_n\|a-a\varphi_n(1_{F_n})\|= 0.\]
To see this, it suffices to consider $a\in A_+^1$. Let $\varepsilon>0$, and choose $\eta>0$ so that $2\eta+\sqrt{3\eta}<\varepsilon$ and $N>0$ so that  $\|a-\varphi_n(\psi_n(a))\|<\eta$ and $\|a^2-\varphi_n(\psi_n(a^2))\|<\eta$ for all $n>N$. 
Then a consequence of Stinespring's Theorem (\cite[Lemma 3.5]{KW}), guarantees that $\|\varphi_n(\psi_n(a)b)-\varphi_n(\psi_n(a))\varphi_n(b)\|<\sqrt{3\eta}$ for any $b\in (F_n)^1$. Applying this to $b=1_{F_n}$ yields
%Then it follows from \cite[Lemma 3.5]{KW} that
\begin{align*}
    \|a-a\varphi_n(1_{F_n})\|&\leq 2\eta + \|\varphi_n(\psi_n(a))-\varphi_n(\psi_n(a))\varphi_n(1_{F_n})\|\\
    &= 2\eta + \|\varphi_n(\psi_n(a)1_{F_n})-\varphi_n(\psi_n(a))\varphi_n(1_{F_n})\|\\
    &<2\eta + \sqrt{3\eta}\\
    &<\varepsilon.
\end{align*}
\end{remarks}

One can always ask for approximately order zero downwards maps, as follows from \cite{BK97} in connection with \cite[Theorem 5]{Voi91}); see  \cite[Theorem 3.1]{BCW} for an explicit statement: 

\begin{theorem}\label{thm: asym order zero cpap}
For a separable $\Cstar$-algebra $A$, the following are equivalent. 
\begin{enumerate}[label=\textnormal{(\roman*)}]
    \item $A$ is nuclear.
    \item $A$ admits a system of c.p.c.\ approximations with approximately order zero downwards maps. 
\end{enumerate}  
\end{theorem}

\begin{lemma}\label{prop: bdd below}
Let $(A\xlongrightarrow{\psi_n}F_n\xlongrightarrow{\varphi_n}A)_n$ be a system of c.p.c.\ approximations of a separable nuclear $\Cstar$-algebra $A$. 
Then the induced map $\Psi\colon A\longrightarrow F_\infty$ given for all $a\in A$ by
\begin{align*}
 \Psi(a)\coloneqq [(\psi_n(a))_n],
\end{align*}
 is a complete order embedding. 

If the system of approximations is summable, then we have  $\Psi(A)=\Lim$, i.e., the image of $\Psi$ agrees with the limit of the associated system in the sense of \emph{Definitions \ref{def: summable}(v)} and \emph{\ref{def: limit}}. 
\end{lemma}

\begin{proof}
Since each $\psi_n$ is c.p.c., it follows that $\Psi$ is a c.p.c. map.
Moreover, since the maps $\psi_n,\varphi_n$ are all c.p.c., 
it follows that \[\textstyle\liminf_n \|\psi^{(r)}_n(a)\|\geq \|a\|\] for all $r\geq 1$ and $a\in \M_r(A)$. Hence $\|\Psi^{(r)}(a)\|\geq \|a\|$ for all $r\geq 0$ and $a\in \M_r(A)$,
and so $\Psi$ is completely isometric and hence a complete order embedding. 

Now we assume moreover that the system is summable with respect to some decreasing sequence $(\varepsilon_n)_n\in \ell^1(\N)^1_+$. 
Let $k\geq 0$ and $x\in F_k$, and set $a\coloneqq \lim_n\varphi_n(\rho_{n,k}(x))$. (Recall that the limit exists by Remark \ref{rmk: summable}(i).) 
Then we have 
\begin{align}\label{eq: approx coh}
    \rho_k(x)&=[(\rho_{m,k}(x))_{m>k}]\\ &=[(\psi_m(\varphi_{m-1}(\rho_{m-1,k}(x))))_{m>k}] \nonumber\\ &=[(\psi_m(a))_m] \nonumber\\
    &=\Psi(a), \nonumber
\end{align}
and so 
\begin{align}
    \textstyle{\bigcup_n} \rho_n(F_n)=\{\Psi\big(\lim_m \varphi_m \circ \rho_{m,n}(x)\big) \mid k\geq 0,\ x\in F_n\}\subset \Psi(A).\label{eq: approx coh'}
\end{align} On the other hand, it follows from summability 
that 
\begin{align*}
    \|\varphi_m\circ \rho_{m,n} \circ \psi_n-\varphi_n\circ\psi_n\|<  \textstyle{\sum}_{j=n+1}^m\varepsilon_j. 
\end{align*}
for all $m>n\geq 0$. Hence for any given $a\in A$ and $\varepsilon>0$, we can find $M>0$ such that for all $m>n>M$, 
\begin{align}\label{eq: summ 2}
&\|\varphi_m \circ \rho_{m,n} \circ \psi_n(a)-a\|\\
&\leq \|\varphi_m \circ \rho_{m,n} \circ \psi_n(a)-\varphi_n\circ\psi_n(a)\|+\|\varphi_n\circ\psi_n(a)-a\|\nonumber \\
&<\varepsilon. \nonumber
    \end{align}
    From this, we conclude that the set $\{\lim_n \varphi_n \circ \rho_{n,k}(x) \mid k\geq 0,\ x\in F_k\}$ is dense in $A$, and so by \eqref{eq: approx coh'}, 
    $\textstyle{\bigcup_n} \rho_n(F_n)$ 
is dense in $\Psi(A)$. Since $\Psi$ is isometric, that means $\Psi(A)=\overline{\textstyle{\bigcup_n} \rho_n(F_n)}=\Lim$.    
\end{proof}

In the situation of \cref{prop: bdd below}, if the system of c.p.c.\ approximations has approximately order zero downwards maps, then the map $\Psi$ clearly is order zero. In this case, since the codomain is $F_\infty$, we obtain refined information about the positive contraction appearing in the structure theorem for order zero maps (\cref{thm: structure thm}).

\begin{proposition}\label{prop: vN quotient}
Let $(A\xlongrightarrow{\psi_n}F_n\xlongrightarrow{\varphi_n}A)_n$ be a system of c.p.c.\ approximations of a separable nuclear $\Cstar$-algebra $A$ with approximately order zero downwards maps, and let $\Psi\colon A\longrightarrow F_\infty$ be the order zero complete order embedding from \emph{\cref{prop: bdd below}}. 
We write $C$ for the $\Cstar$-algebra generated by $\Psi(A)$ in $F_\infty$ and $J_C=\{b\in F_\infty \mid b C\cup C b\subset C\}$ for the idealizer of $C$ in $F_\infty$, and we denote by $h$ the element in $\mathcal{M}(C)$ provided by the structure theorem for order zero maps; cf.\ \emph{Theorem \ref{thm: structure thm}}. 

Then there exists $\bar{h}\in (J_C)_+^1 \subset F_\infty$ so that $\bar{h}c=hc$ 
for all $c\in C$. In particular, for all $a,b\in A$, 
 \begin{align}
    \bar{h}\Psi(a)&=\Psi(a)\bar{h}, \label{eq: h1}  \\
    \|\bar{h}\Psi(a)\|&=\|\Psi(a)\|=\|a\|,\ \text{ and } \label{eq: h2} \\
    \bar{h}\Psi(ab)&=\Psi(a)\Psi(b). \label{eq: h3}
 \end{align}
 Moreover, \eqref{eq: h1}, \eqref{eq: h2}, and \eqref{eq: h3} still holds when replacing $\bar{h}$ by any $\bar{e}\in (F_\infty)_+^1$ such that $\bar{e} \Psi(a)=\bar{h}\Psi(a)$ for all $a\in A$.
\end{proposition}

\begin{proof}
Since $C$ is an ideal in $J_C$, there is a unique $^*$-homorphism $\theta\colon J_C\longrightarrow \mathcal{M}(C)$ extending the canonical embedding $C \longrightarrow\mathcal{M}(C)$ so that $\theta(b)c=bc$ and $c\theta(b)=cb$ for all $b\in J_C$ and $c\in C$. We claim that $\theta$ is surjective. 
Indeed, let $\pi\colon \prod F_n\longrightarrow F_\infty$ denote the quotient map. Since $C$ is separable, we can find a separable sub-$\Cstar$-algebra $D\subset \prod F_n$ with $\pi(D)=C$. Though $D$ may be degenerate in $\prod F_n$ (in the sense that its support projection is not $1_{\prod F_n}$), it is non-degenerate in its ultra-weak closure $\overline{D}^{\sigma\text{-wk}}\subset \prod F_n$ (in the same sense), and so \cite[Proposition 2.4]{APT} tells us we can 
identify $\mathcal{M}(D)$ with $\{b\in \overline{D}^{\sigma\text{-wk}} \mid b D\cup D b\subset D\}\subset \prod F_n$, 
%find a copy of $\mathcal{M}(D)$ with $D\subset \mathcal{M}(D) 
%%\subset \overline{D}^{\theta\text{-wk}} 
%\subset \prod F_n$, and so %(Indeed, taking a faithful non-degenerate representation of $\overline{D}^{\theta\text{-wk}}$, this follows from \cite[Proposition 3.12.3]{Ped2}.) 
and with this identification, we have $\pi(\mathcal{M}(D))\subset J_C$. 
Since $D$ is separable and $\pi|_D\colon D\longrightarrow C$ is a surjection, 
\cite[Theorem 4.2]{APT} guarantees a surjective $^*$-homomorphism $\hat{\pi}\colon \mathcal{M}(D)\longrightarrow \mathcal{M}(C)$ so that for all $d\in D$ and $x\in \mathcal{M}(D)$ we have 
$\hat{\pi}(x)\pi(d)=\pi(xd)$ and $\pi(d)\hat{\pi}(x)=\pi(dx)$. All this together gives us the following diagram, which we claim commutes. 
\[
\begin{tikzcd}
D\arrow[two heads, swap]{d}{\pi|_D}\arrow[r,phantom, "\rotatebox{360}{$\subseteq$}", description] & \mathcal{M}(D)\arrow[swap]{d}{\pi|_{\mathcal{M}(D)}}\arrow[two heads]{rd}{\hat{\pi}} & \\
C \arrow[r,phantom, "\rotatebox{360}{$\subseteq$}", description] & J_C \arrow{r}{\theta} & \mathcal{M}(C)
\end{tikzcd}
\]
To see that $\hat{\pi}=\theta\circ\pi|_{\mathcal{M}(D)}$, let $x\in \mathcal{M}(D)$ and $c\in C$, and choose $d\in D$ such that $\pi(d)=c$. Then
\begin{align*}
\theta(\pi(x))c=\pi(x)c=\pi(x)\pi(d)=\pi(xd)=\hat{\pi}(x)\pi(d)=\hat{\pi}(x)c, 
\end{align*}
and likewise $c\theta(\pi(c))=c\hat{\pi}(x)$. 
Hence $\hat{\pi}=\theta\circ\pi|_{\mathcal{M}(D)}$, and $\theta$ is surjective as desired. 
 
Now let $\bar{h}\in (J_C)_+^1$ be a lift of $h$. 
Then we have for each $c\in C$ 
\begin{align*}
    \bar{h}c=\theta(\bar{h})c=hc\hspace{.5 cm} \text{ and } \hspace{.5 cm}
    c\bar{h}=c\theta(\bar{h})=ch.
\end{align*}
It then follows from our choice of $h$ 
that for all $a,b\in A$, 
\begin{align*}
\bar{h}\Psi(ab)&=h\Psi(ab)=\Psi(a)\Psi(b),\ \text{ and }\\
\bar{h}\Psi(a)&=h\Psi(a)=\Psi(a)h=\Psi(a)\bar{h}. 
\end{align*}
Since $\Psi$ is isometric, $\|a\|^2h=\|\Psi(aa^*)\|h\geq \Psi(aa^*)$ for all $a\in A$ (see Remark \ref{cor: order zero for unital}(i)), and so
\begin{align*}
    \|a\|^4&=\|\Psi(a)\|^4\\
    &=\|\Psi(a)^*\Psi(a)\Psi(a)^*\Psi(a)\|\\ 
    &\leq \|\Psi(a)^*\Psi(aa^*)\Psi(a)\|\\
    &\leq \|\Psi(a)^*\|a\|^2h\Psi(a)\|\\
    &\leq \|a\|^3\|h\Psi(a)\|.
\end{align*}
In particular, for all $a\in A$
\begin{align*}
   \|\bar{h}\Psi(a)\|= \|h\Psi(a)\|=\|\Psi(a)\|=\|a\|. 
\end{align*}

Finally, suppose $\bar{e}\in (F_\infty)_+^1$ and $\bar{e} \Psi(a)=\bar{h}\Psi(a)=h\Psi(a)$ for all $a\in A$. Then for each $a\in A$, 
\[\|\bar{e} \Psi(a)-\Psi(a) \bar{e}\|=\|h\Psi(a)-\Psi(a) \bar{e}\|=\|h\Psi(a)^*-\bar{e} \Psi(a)^*\|=0,\]
 and it follows that $\Psi(a)\bar{e} = \Psi(a)\bar{h} = \Psi(a)h$ for all $a\in A$ as well.  
Since this extends to all $c\in C$, it follows that $\bar{e}\in (J_C)_+^1$ and $\theta(\bar{e})=h$. Since $\bar{e}$ is also a lift of $h$, the preceding arguments establish \eqref{eq: h1}, \eqref{eq: h2}, and \eqref{eq: h3} for $\bar{e}$.  
\end{proof}

Now we are ready to prove the main result of this section.

\begin{theorem}\label{thm: main 2}
Let $(A\xlongrightarrow{\psi_n}F_n\xlongrightarrow{\varphi_n}A)_n$ be a system of c.p.c.\ approximations of a separable nuclear $\Cstar$-algebra $A$ with approximately order zero downwards maps. After possibly passing to a subsystem of approximations, the associated system $(F_n,\rho_{n+1,n})_n$ is a $\CPCstar$-system in the sense of \emph{\cref{def: cpcstar}}, and the map $\Psi\colon A\longrightarrow F_\infty$ from \emph{\cref{prop: bdd below}} is an order zero complete order embedding with $\Psi(A)=\Lim $.

Upon composing $\Psi$ with the inverse of $\Theta$ from \emph{Proposition \ref{thm: order zero coi}}, we obtain an isomorphism $\Theta^{-1}\circ\Psi\colon A\longrightarrow \Cstar_{\sbt}\big(\Lim\big)$. In particular, any separable nuclear $\Cstar$-algebra is $^*$-isomorphic to a $\CPCstar$-limit. 
\end{theorem}

\begin{proof}
As we saw in Remark \ref{rmk: summable}(ii), after possibly passing to a subsystem, we may assume that the system of approximations (and any further subsystem) is summable.
    
 Let $\Psi\colon A\longrightarrow F_\infty$ be the induced order zero complete order embedding from \cref{prop: bdd below}, $h$ be the element in $\mathcal{M}(\Cstar(\Psi(A)))$ from the structure theorem for order zero maps (\cref{thm: structure thm}), and $\bar{h}\in (F_\infty)_+^1$ be an element guaranteed by \cref{prop: vN quotient} so that $\bar{h}\Psi(a)=h\Psi(a)$ for all $a\in A$. 
We fix a lift $(h_n)_n\in (\prod F_n)_+^1$ of $\bar{h}$ and note that since 
 \begin{align*}\|\Psi(b)\Psi(a)-\bar{h}\Psi(a)\|&\overset{\eqref{eq: h3}}{=}\|\bar{h}\Psi(ba)-\bar{h}\Psi(a)\|\\ 
 &\overset{\phantom{\eqref{eq: h3}}}{=}\|\bar{h}\Psi(ba-a)\|\\
 &\overset{\eqref{eq: h2}}{=} \|ba-a\|\end{align*} 
 for all $a,b\in A$, 
we have that for any $a,b\in A$ and $\varepsilon>0$ there exists $M>0$ so that for all $m\geq M$
\begin{align}\label{eq: approx 1}
    \|\psi_m(b)\psi_m(a)-h_m\psi_m(a)\|< \|ba-a\|+ \varepsilon.
\end{align}

Now we are ready to determine an appropriate subsystem of approximations. Fix a countable dense subset $\{a_k\}_k\subset A^1$ and a strictly decreasing sequence $(\varepsilon_n)_n\in \ell^1(\N)_+^1$.
  By Remark \ref{rmk: summable}(iii) we may choose $j_0>0$ so that 
    \begin{align*}
    \|\varphi_{j_0}(1_{F_{j_0}})a_0-a_0\|<\varepsilon_0/2.\end{align*} 
Then we can choose $j_1>j_0$ so that for all $m\geq j_1$ and $i\in \{1,2\}$
\begin{align*}
    \|\psi_m(\varphi_{j_0}(1_{F_{j_0}}))\psi_m(a_0)-h_m\psi_m(a_0)\|&\overset{\eqref{eq: approx 1}}{<}\|\varphi_{j_0}(1_{F_{j_0}})a_0-a_0\|+ \varepsilon_0/2 \\
    &\overset{\phantom{\eqref{eq: approx 1}}}{<}\varepsilon_0,\\
    \max_{0\leq k\leq 1}\|\varphi_{j_1}(1_{F_{j_1}})a_k-a_k\|&\overset{\phantom{\eqref{eq: approx 1}}}{<}\varepsilon_1/2, \text{ and}\\
     \|\varphi_{j_1}(\psi_{j_1}(\varphi_{j_0}(1_{F_{j_0}})^i))-\varphi_{j_0}(1_{F_{j_0}})^i\|&\overset{\phantom{\eqref{eq: approx 1}}}{<}\varepsilon_1^2/3. 
\end{align*}
For all $n>0$ we may inductively choose $j_n>j_{n-1}$ so that for all $m\geq j_n$ and $i\in \{1,2\}$,
\begin{align}
 \max_{0\leq k <n}  \|\psi_m(\varphi_{j_{n-1}}(1_{F_{j_{n-1}}}))\psi_m(a_k)-h_m\psi_m(a_k)\|&<\varepsilon_{n-1}, \label{eq: first approx1}\\
 \max_{0\leq k\leq n} \|\varphi_{j_n}(1_{F_{j_n}})a_k-a_k\|&<\varepsilon_n/2,  \text{ and} \label{eq: first approx2}\\
    \|\varphi_{j_n}(\psi_{j_n}(\varphi_{j_{n-1}}(1_{F_{j_{n-1}}})^i))-\varphi_{j_{n-1}}(1_{F_{j_{n-1}}})^i\|&<\varepsilon_n^2/3.  
    \label{eq: first approx3}
\end{align}
 
Let $\hat{\Psi}\colon A\longrightarrow \prod F_{j_n}/\textstyle{\bigoplus} F_{j_n}$ denote the map induced by the subsystem $(A\xlongrightarrow{\psi_{j_n}}F_{j_n}\xlongrightarrow{\varphi_{j_n}}A)_n$. 
 Then $\hat{\Psi}$ is still an order zero complete order embedding, and 
 for $\hat{\bar{h}}=[(h_{j_n})_n]\in \prod F_{j_n}/\textstyle{\bigoplus} F_{j_n}$ we still have $\hat{\bar{h}}\hat{\Psi}(a)=\hat{h}\hat{\Psi}(a)$ for all $a\in A$, where $\hat{h}\in \mathcal{M}(\Cstar(\hat{\Psi(A)}))$ is the element guaranteed by \cref{thm: structure thm}.
 Thus we may pass to the (summable) subsystem $(A\xlongrightarrow{\psi_{j_n}}F_{j_n}\xlongrightarrow{\varphi_{j_n}}A)_n$, drop the subscripts, and write $\Psi$, $h$, and $\bar{h}$ for $\hat{\Psi}$, $\hat{h}$, and $\hat{\bar{h}}$, respectively. 
Now \eqref{eq: first approx1}--\eqref{eq: first approx3} say that for any $n>0$, we have for for $m> n$ and $i\in\{1,2\}$
 \begin{align}
  \max_{0\leq k \leq n} \|\psi_m(\varphi_{n}(1_{F_{n}}))\psi_m(a_k)-h_m\psi_m(a_k)\|&<\varepsilon_{n}, 
  \label{eq: second approx1}\\
   \max_{0\leq k\leq n}\|\varphi_{n}(1_{F_{n}})a_k-a_k\|&<\varepsilon_n/2, \text{ and}\label{eq: second approx2}\\
  \|\varphi_n(\psi_n(\varphi_{n-1}(1_{F_{n-1}})^i))-\varphi_{n-1}(1_{F_{n-1}})^i\|&<\varepsilon_n^2/3. 
  \label{eq: second approx3}
\end{align}
We aim to show that the associated system $(F_n,\rho_{n+1,n})_n$ as in Definition \ref{def: summable}(v) is $\CPCstar$.\\

\noindent {\bf Claim 1:} 
For all $a\in A$, we have $[(\rho_{n+1,n}(1_{F_n}))_n] \Psi(a)=\bar{h}\Psi(a)$, and \eqref{eq: h1}, \eqref{eq: h2}, and \eqref{eq: h3} also hold for $[(\rho_{n+1,n}(1_{F_n}))_n]$ in place of $\bar{h}$. 

We check the claim on an element $a_k$ from our fixed dense subset of $A^1$. Since
\begin{align}
&\|\rho_{n,n-1}(1_{F_{n-1}}) \psi_n(a_k)-h_n\psi_n(a_k)\| \label{eq: Claim 1} \\
&\overset{\phantom{\eqref{eq: second approx1}}}{=}\|\psi_n(\varphi_{n-1}(1_{F_{n-1}})) \psi_n(a_k)-h_n\psi_n(a_k)\| \nonumber\\
&\overset{\eqref{eq: second approx1}}{<}\varepsilon_{n-1}\nonumber 
\end{align}
for all $n>k$, it follows that $[(\rho_{n+1,n}(1_{F_n}))_n]\Psi(a)=\bar{h}\Psi(a)$ for all $a\in A$, which establishes the claim by \cref{prop: vN quotient} with $\bar{e}=[(\rho_{n+1,n}(1_{F_n}))_n]$. 

\medskip

\noindent {\bf Claim 2:} For any $a\in A$ and $\varepsilon>0$, there exists an $M>0$ so that for $m>n>M$ 
\begin{align*}
\|\rho_{m,n}(1_{F_n})\psi_m(a)-\rho_{m,m-1}(1_{F_{m-1}})\psi_m(a)\|<\varepsilon.\end{align*}

Again, we check the claim on some $a_k$ from our fixed dense subset of $A^1$. Fix $\varepsilon>0$. It follows from summability 
(see Remark \ref{rmk: summable}(i)) that we may choose $M>k$ with $\varepsilon_M<\varepsilon/3$ so that $\|\rho_{m,n}-\psi_m\circ\varphi_n\|<\varepsilon/3$ for all $m>n>M$. 
Then it follows from \eqref{eq: second approx1} and \eqref{eq: Claim 1} that for all $m>n>M$, 
\begin{align*}
    &\|\rho_{m,n}(1_{F_n})\psi_m(a_k)-\rho_{m,m-1}(1_{F_{m-1}})\psi_m(a_k)\|\\
    &<  \|\psi_m(\varphi_{n}(1_{F_{n}}))\psi_m(a_k)-\rho_{m,m-1}(1_{F_{m-1}})\psi_m(a_k)\| + \varepsilon/3\\
     &\leq  
     \|\psi_m(\varphi_{n}(1_{F_{n}}))\psi_m(a_k)-h_m\psi_m(a_k)\| \\
    & \phantom{{}\leq {}
     } 
     + \|h_m\psi_m(a_k)-\rho_{m,m-1}(1_{F_{m-1}})\psi_m(a_k)\| +\varepsilon/3 
     \\
 &<  \varepsilon_{n}+\varepsilon_{m-1}+\varepsilon/3\\
 &<\varepsilon.
\end{align*}

\medskip

\noindent {\bf Claim 3:} For any $a,b\in A$ and $\varepsilon>0$, there exists an $M>0$ so that $m> n>M$
\begin{align*}
    \|\psi_m(a)\psi_m(b)-\rho_{m,n}(1_{F_n})\psi_m(ab)\|&<\varepsilon. 
\end{align*}

Let $a,b\in A^1$ and $\varepsilon>0$. Claim 1 tells us that $[(\rho_{n+1,n}(1_{F_n}))_n] \Psi(ab)=\Psi(a)\Psi(b)$, and so there is an $M_0>0$ so that for all $m>M_0$ 
\begin{align*}
    \|\psi_m(a)\psi_m(b)-\rho_{m,m-1}(1_{F_{m-1}})\psi_m(ab)\|
    <\varepsilon/2.
    \end{align*}
    Then this with Claim 2 (for $ab$ in place of $a\in A$ and $\varepsilon/2$ in place of $\varepsilon$) tells us that there exists $M>M_0$ so that for all $m> n>M$, 

\begin{align*}
&\|\psi_m(a)\psi_m(b)-\rho_{m,n}(1_{F_n})\psi_m(ab)\| \\ &<\|\rho_{m,m-1}(1_{F_{m-1}})\psi_m(ab)-\rho_{m,n}(1_{F_n})\psi_m(ab)\| + \varepsilon/2\\
&<\varepsilon.
\end{align*}

\noindent {\bf Claim 4:}
For any $a\in A$ and $\varepsilon>0$, there exists an $M>0$ so that for all $m>n>M$, 
\begin{align*}
    \|\rho_{m,n}\big(\rho_{n,n-1}(1_{F_{n-1}})\psi_n(a)\big)-\psi_m(a)\|&<\varepsilon. 
\end{align*}

Again, we check the claim on some $a_k$ from our fixed dense subset of $A^1$. Fix $1>\varepsilon>0$. 
By \eqref{eq: second approx3} and \cite[Lemma 7.11]{ER02}, for each $n>0$

\begin{align*}
    &\|\varphi_n\big(\rho_{n,n-1}(1_{F_{n-1}})\psi_n(a_k)\big)-\varphi_n(\rho_{n,n-1}(1_{F_{n-1}}))\ \varphi_n(\psi_n(a_k))\|\\ 
   & <\big(3\varepsilon_n^2/3\big)^{1/2}\\
   &=\varepsilon_n.
\end{align*}
We now can choose $M>k$ with $\varepsilon_M<\varepsilon/6$ so that for all $m>n>M$,

\begin{align}
    &\|\varphi_n\big(\rho_{n,n-1}(1_{F_{n-1}})\psi_n(a_k)\big)-\varphi_n(\rho_{n,n-1}(1_{F_{n-1}}))\ \varphi_n(\psi_n(a_k))\|\label{eq: claim 4.1}\\
    &<\varepsilon_n \nonumber\\
    &<\varepsilon/6,\nonumber 
    \end{align}
    \begin{align}
    \|\varphi_n(\psi_n(a_k))-a_k\|&<\varepsilon/6,\label{eq: claim 4.2}
    \end{align}
    and
    \begin{align}
    \|\rho_{m,n}(\psi_n(a_k))-\psi_m(a_k)\|&\leq \|\varphi_{m-1}\circ\rho_{m-1,n}(\psi_n(a_k))-a_k\| \label{eq: claim 4.3}\\
    &<\varepsilon/2. \nonumber
\end{align}

Then we have for all $m>n>M$, 

\begin{align*}
    & \|\rho_{m,n}\big(\rho_{n,n-1}(1_{F_{n-1}})\psi_n(a_k)\big)-\psi_m(a_k)\|\\  
    &\overset{\phantom{\eqref{eq: second approx3}}}{\leq} 
    \|\rho_{m,n}\big(\rho_{n,n-1}(1_{F_{n-1}})\psi_n(a_k)\big)-\rho_{m,n+1}(\psi_{n+1}(a_k))\|\\
    &\phantom{{}\overset{\eqref{eq: second approx3}}{\leq}{}}+\|\rho_{m,n+1}(\psi_{n+1}(a_k))-\psi_m(a_k)\|
    \\
&\overset{\eqref{eq: claim 4.3}}{<} \varepsilon/2 + \|\rho_{m,n}\big(\rho_{n,n-1}(1_{F_{n-1}})\psi_n(a_k)\big)-\rho_{m,n+1}(\psi_{n+1}(a_k))\|\\
    &\overset{\phantom{\eqref{eq: second approx3}}}{=}  \varepsilon/2 + \|\rho_{m,n+1}\circ\psi_{n+1}\big(\varphi_n\big(\rho_{n,n-1}(1_{F_{n-1}})\psi_n(a_k)\big)-a_k\big)\|\\
    &\overset{\phantom{\eqref{eq: second approx3}}}{\leq}  \varepsilon/2 + \|\varphi_n\big(\rho_{n,n-1}(1_{F_{n-1}})\psi_n(a_k)\big)-a_k\|\\
  &\overset{\eqref{eq: claim 4.1}}{<} \varepsilon/2 +\varepsilon/6 + \|\varphi_n(\rho_{n,n-1}(1_{F_{n-1}}))\  \varphi_n(\psi_n(a_k))-a_k\|\\
  &\overset{\eqref{eq: claim 4.2}}{<} 2\varepsilon/3 + \varepsilon/6 + \|\varphi_n(\rho_{n,n-1}(1_{F_{n-1}}))a_k-a_k\|\\
   &\overset{\phantom{\eqref{eq: second approx3}}}{\leq}  5\varepsilon/6+ \|\varphi_n(\rho_{n,n-1}(1_{F_{n-1}}))a_k-\varphi_{n-1}(1_{F_{n-1}})a_k\|\\ 
   &\phantom{{}\overset{\eqref{eq: second approx3}}{\leq}{}}+ \|\varphi_{n-1}(1_{F_{n-1}})a_k-a_k\|\\
   &\overset{\eqref{eq: second approx3}}{<} 5\varepsilon/6+ \varepsilon_n^2/3+\|\varphi_{n-1}(1_{F_{n-1}})a_k-a_k\|\\
   &\overset{\eqref{eq: second approx2}}{<} 5\varepsilon/6+ \varepsilon_n^2/3+\varepsilon_{n-1}/2\\ &\overset{\phantom{\eqref{eq: second approx3}}}{<}\varepsilon.
\end{align*} 

\xdef\tpd{\the\prevdepth}
This finishes the proof of the claim.

\bigskip 

Now we are ready to verify that $(F_n,\rho_{n+1,n})_n$ is a $\CPCstar$-system.  
Let $k\geq 0$, $x,y\in F^1_k$, and $\varepsilon>0$. Define $a\coloneqq \lim_n\varphi_n(\rho_{n,k}(x))\in A$ and $b\coloneqq \lim_n\varphi_n(\rho_{n,k}(y))\in A$. (Recall that these limits exist by Remark \ref{rmk: summable}(i).) We then in particular have
\begin{align*}
\lim_{n}\|\rho_{n,k}(x)\rho_{n,k}(y)-\psi_n(a)\psi_n(b)\|= 0.
\end{align*}
Combining this with Claims 3 and 4, we can choose $M>0$ so that for all $m>n\geq M$,
\begin{align*}
 \|\rho_{m,k}(x)\rho_{m,k}(y)-\rho_{m,n}(1_{F_n})\psi_m(ab)\|&<\varepsilon/4\ \text{ and} 
 \\
 \|\rho_{m,n}\big(\rho_{n,n-1}(1_{F_{n-1}})\psi_n(ab)\big)-\psi_m(ab)\|&<\varepsilon/2. 
\end{align*}
Using 
these two approximations, we compute for $m>n,l>M$,

\begin{align*}
    &\|\rho_{m,l}(1_{F_l})\rho_{m,n}\big(\rho_{n,k}(x)\rho_{n,k}(y)\big)-\rho_{m,k}(x)\rho_{m,k}(y)\|\\ 
    &\leq
    \|\rho_{m,l}(1_{F_l})\  \rho_{m,n}\big(\rho_{n,k}(x)\rho_{n,k}(y)-\rho_{n,n-1}(1_{F_{n-1}})\psi_n(ab)\big)\| \\
    &\phantom{{}\leq{}} + \|\rho_{m,l}(1_{F_l})\ \rho_{m,n}\big(\rho_{n,n-1}(1_{F_{n-1}})\psi_n(ab)\big)-\rho_{m,l}(1_{F_l})\  \psi_m(ab)\| \\
   &\phantom{{}\leq{}}+\|\rho_{m,l}(1_{F_l})\psi_m(ab)-\rho_{m,k}(x)\rho_{m,k}(y)\|
   \\
   &\leq  
   \|\rho_{n,k}(x)\rho_{n,k}(y)-\rho_{n,n-1}(1_{F_{n-1}})\psi_n(ab)\| \\
   &\phantom{{}\leq{}} + \|\rho_{m,n}\big(\rho_{n,n-1}(1_{F_{n-1}})\psi_n(ab)\big)-\psi_m(ab)\| \\
   &\phantom{{}\leq{}}+\|\rho_{m,l}(1_{F_l})\psi_m(ab)-\rho_{m,k}(x)\rho_{m,k}(y)\| 
   \\
&< \varepsilon. 
\end{align*}

Finally, since our subsystem is still summable, $\Psi(A)=\Lim $ by  \cref{prop: bdd below}, so by  \cref{prop: CPC*-algebras'} $A$ is $^*$-isomorphic to a $\CPCstar$-limit. 
    \end{proof}

    \begin{remark}\label{rmk: Upsilon}
    If $A$ is a separable unital nuclear $\Cstar$-algebra with a summable system $(A\xlongrightarrow{\psi_n}F_n\xlongrightarrow{\varphi_n}A)_n$ of c.p.c.\ approximations with approximately order zero downwards maps, then $(\varphi_n(1_{F_n}))_n$ converges to $1_A$, and hence $\Psi(1_A)=\lim_n\rho_n(1_{F_n})=e$. In this case, the associated system $(F_n,\rho_{n+1,n})_n$ will automatically be a $\CPCstar$-system even without passing to a subsystem. 
    \end{remark}

We close this section with a description of the product of a nuclear $\Cstar$-algebra in terms of an associated $\CPCstar$-system, interpreting \cref{prop: CPC*-algebras'} along the lines of \cref{thm: main 1}.
\begin{corollary}
Let $(A\xlongrightarrow{\psi_n}F_n\xlongrightarrow{\varphi_n}A)_n$ be a summable system of c.p.c.\ approximations of a separable nuclear $\Cstar$-algebra $A$ with approximately order zero downwards maps so that the associated system $(F_n,\rho_{n+1,n})_n$ is a $\CPCstar$-system. Let $\Psi\colon A\longrightarrow F_\infty$ be the induced order zero complete order embedding. Then for any $k\geq 0$, $x,y\in F_k$, and $a=\lim_n\varphi_n(\rho_{n,k}(x))$, $b=\lim_n\varphi_n(\rho_{n,k}(y))\in A$, we have 
$\Psi(ab)=\Psi(a)\sbt\Psi(b)=\lim_n\rho_n(\rho_{n,k}(x)\rho_{n,k}(y))$. 
\end{corollary}

 \begin{proof}
 Like in \eqref{eq: approx coh} we have $\Psi(a)=\rho_k(x)$ and $\Psi(b)=\rho_k(y)$, and so by Propositions \ref{thm: main 1} and \ref{thm: order zero coi}, 
 \begin{align*}
     \Psi(ab)&=\Theta^{-1}(\Psi(ab))\\
     &= \big( \Theta^{-1}\circ\Psi(a)\big)\sbt \big(\Theta^{-1}\circ\Psi(b)\big)\\ &=\rho_k(x)\sbt \rho_k(y)\\ &=\lim_n\rho_n(\rho_{n,k}(x)\rho_{n,k}(y)). \qedhere
 \end{align*}
 \end{proof}

\section{NF systems and $\CPCstar$-systems}\label{sect: NF}
\noindent We conclude by comparing our $\CPCstar$-systems with Blackadar and Kirchberg's NF systems from \cite{BK97} and establish that the latter may indeed be regarded as a special case of the former.

\begin{definition}[{\cite[Definitions 2.1.1 and 5.2.1]{BK97}}]\label{def: NF}
A system $(F_n,\rho_{n+1,n})_n$ consisting of finite dimensional $\Cstar$-algebras and c.p.c.\ maps as in Definition \ref{def: limit} is called an \emph{NF system} if the maps $\rho_{m,n}$ are \emph{asymptotically multiplicative}, in the sense that for all $k\geq 0$, $x,y\in F_k$, and $\varepsilon>0$, there exists $M>k$ so that for all $m>n>M$, 
\[ \|\rho_{m,n}(\rho_{n,k}(x)\rho_{n,k}(y))-\rho_{m,k}(x)\rho_{m,k}(y)\|<\varepsilon.\]
The limit of this system (in the sense of \cref{def: limit}) is a $\Cstar$-algebra with multiplication given for all $k\geq 0$ and $x,y\in F_k$ by
\[\rho_k(x)\rho_k(y)=\lim_n \rho_n(\rho_{n,k}(x)\rho_{n,k}(y)).\]
A $\Cstar$-algebra that is $^*$-isomorphic to the inductive limit of an NF system is called an \emph{NF algebra}.
\end{definition}

\begin{remark}
A unital $\CPCstar$-system is NF, and if $(F_n,\rho_{n+1,n})_n$ is an NF system such that for any $\varepsilon>0$ there exists $M>0$ so that for all $m>n>M$,
\begin{align*}
    \|\rho_{m,n}(1_{F_n})-1_{F_m}\|<\varepsilon,
\end{align*}
then it is also a $\CPCstar$-system. With \cref{cor: NF has cpcstar subsystem} below, one can drop this condition upon passing to a subsystem. 
\end{remark}

The following excerpt from \cite[Theorem 5.2.2]{BK97} characterizes nuclear and quasidiagonal $\Cstar$-algebras in terms of NF systems. In comparison with \cref{thm: main 2}, here one asks for approximately multiplicative downwards maps as opposed to just approximately order zero ones; this corresponds to adding quasidiagonality to the nuclearity hypothesis.

\begin{theorem}
For a separable $\Cstar$-algebra $A$ the following are equivalent:
\begin{enumerate}[label=\textnormal{(\roman*)}]
    \item $A$ is an NF algebra. 
    \item $A$ is nuclear and quasidiagonal.
    \item There exists a system $(A\xlongrightarrow{\psi_n} F_n \xlongrightarrow{\varphi_n} A)_n$ of c.p.c.\ approximations with approximately multiplicative downwards maps.
\end{enumerate}
\end{theorem}

\begin{theorem}\label{cor: NF has cpcstar subsystem}
Any NF system admits a $\CPCstar$-subsystem.
\end{theorem}

\begin{proof}
Let $(F_n,\rho_{n+1,n})_n$ be an NF system. From this we form another NF system $(B_n,\varphi_{n+1,n})$ (with a different limit) just as in \cite[Proposition 5.1.3]{BK97} 
by taking $B_n=\bigoplus_{j=0}^n F_j$ and $\varphi_{n+1,n}\coloneqq \id_{B_n}\oplus \rho_{n+1,n}$:

\[
\begin{tikzcd}
B_0\arrow[d,phantom, "\rotatebox{90}{$=$}", description]
\arrow{r}{\varphi_{1,0}} & B_1 \arrow[d,phantom, "\rotatebox{90}{$=$}", description]
\arrow{r}{\varphi_{2,1}} & B_2 \arrow[d,phantom, "\rotatebox{90}{$=$}", description]
&\hdots 
\\
F_0\arrow[swap]{dr}{\rho_{1,0}}\arrow{r}{\id}& F_0\arrow[d,phantom, "\rotatebox{0}{$\bigoplus$}", description] \arrow{r}{\id} & F_0\arrow[d,phantom, "\rotatebox{0}{$\bigoplus$}", description]
&
\hdots\\
& F_1\arrow[swap]{dr}{\rho_{2,1}}\arrow{r}{\id} & F_1\arrow[d,phantom, "\rotatebox{0}{$\bigoplus$}", description]
&
\hdots \\
&& F_2 
& 
\hdots
\end{tikzcd} 
\] 
Then for each $k\geq 0$, $\hat{x}=(x_0,...,x_k)\in B_k$, and $n>k$, we have 
\[\varphi_{n,k}(\hat{x})= \hat{x} \oplus \oplus_{j=k+1}^n \rho_{j,k}(x_k).\] 

We claim we can find a subsystem $(B_{n_j},\varphi_{n_j+1,n_j})_j$ which is $\CPCstar$. Set $B\coloneqq{\overset{\tikz \draw [arrows = {->[harpoon]}] (-1,0) -- (0,0);}{(B_n,\varphi_n)}_n}$. Then each $\varphi_{n+1,n}$ is a complete order embedding, and hence so are the induced c.p.c.\ maps $\varphi_n\colon B_n\longrightarrow  B\subset B_\infty$. 
Using Arveson's Extension Theorem (cf.\ \cite[Lemma 5.1.2]{BK97}), we extend each c.p.c.\ contraction $\varphi_n^{-1}\colon \varphi_n(B_n)\longrightarrow  B_n$ to a c.p.c.\ map $\psi_n\colon B\longrightarrow  B_n$ with $\psi_n\circ\varphi_n=\id_{B_n}$.  
Since the union $\bigcup_n \varphi_n(B_n)$ is nested, 
$(\varphi_m\circ\psi_m)|_{\varphi_n(B_n)}=\id_{\varphi_n(B_n)}$ for all $m>n\geq 0$, and so $(B\xlongrightarrow {\psi_n}B_n\xlongrightarrow {\varphi_n}B)_n$ is a system of c.p.c.\ approximations of $B$. Note that the associated system $(B_n,\psi_{n+1}\circ\varphi_n)_n$ is exactly $(B_n,\varphi_{n+1,n})_n$: Indeed, for any $n\geq 0$ and $x\in B_n$, we have for all $m>n$
\begin{align}
    &(\psi_m\circ\varphi_{m-1}\circ\psi_{m-1}\circ\hdots\varphi_{n+1}\circ\psi_{n+1}) (\varphi_n(x))\label{eq: jump} \\ 
    &=\psi_m(\varphi_n(x)) \nonumber\\ &=\psi_m(\varphi_m(\varphi_{m,n}(x)))\nonumber \\ &=\varphi_{m,n}(x). \nonumber
\end{align}
Moreover, the downwards maps of our system $(B\xlongrightarrow {\psi_n}B_n\xlongrightarrow {\varphi_n}B)_n$ of c.p.c.\ approximations are approximately multiplicative. To see this, it suffices to consider $\varphi_k(x),\varphi_k(y)$ for some $k\geq 0$ and $x,y\in B_k$. Fix $\varepsilon>0$. Since $\bigcup_n \varphi_n(B_n)$ is dense in $B$, we may choose $N>k$ and $z\in B_N$ so that $\|\varphi_N(z)-\varphi_k(x)\varphi_k(y)\|<\varepsilon/3$. Since the system is NF, we have 
\begin{align*}
    \varphi_k(x)\varphi_k(y)&=\lim_n\varphi_n\big(\varphi_{n,k}(x)\varphi_{n,k}(y)\big)\overset{\eqref{eq: jump}}{=}\lim_n\varphi_n\big(\psi_n(\varphi_k(x))\psi_n(\varphi_k(y))\big),
\end{align*}
and so we may choose $M>N$ so that \[\|\varphi_k(x)\varphi_k(y)-\varphi_m\big(\psi_m(\varphi_k(x))\psi_m(\varphi_k(y))\big)\|<\varepsilon/3\] for all $m>M$. 
Since each $\varphi_n$ is isometric, we have for all $m>M>N$
\begin{align*}
    &\|\psi_m(\varphi_k(x)\varphi_k(y))-\psi_m(\varphi_k(x))\psi_m(\varphi_k(y))\|\\
    &< \|\psi_m(\varphi_N(z))-\psi_m(\varphi_k(x))\psi_m(\varphi_k(y))\| + \varepsilon/3 \\
  &=\|\varphi_m\big(\psi_m(\varphi_N(z))\big)-\varphi_m\big(\psi_m(\varphi_k(x))\psi_m(\varphi_k(y)))\big)\| + \varepsilon/3 \\
   &=\|\varphi_N(z)-\varphi_m\big(\psi_m(\varphi_k(x))\psi_m(\varphi_k(y))\big)\| + \varepsilon/3 \\
  &<\varepsilon.
\end{align*}
Since the downwards maps are approximately multiplicative, they are also approximately order zero, and so \cref{thm: main 2} guarantees that the associated system $(B_n,\psi_{n+1}\circ\varphi_n)_n=(B_n,\varphi_{n+1,n})_n$, has a $\CPCstar$-subsystem $(B_{n_j},\varphi_{n_{j+1},n_j})_j$. 
\vspace{.25 cm}

We claim that the system $(F_{n_j},\rho_{n_{j+1},n_j})_j$ is also $\CPCstar$. Since this will hold without passing to a further subsystem, at this point we may drop the subscripts for ease of notation.
Fix $k\geq 0$, $x,y\in F_k$, and $\varepsilon>0$. Let $\hat{x}=(0,...,0,x), \hat{y}=(0,...,0,y)\in B_k$, and choose $M>k$ so that for all $m>n,l>M$ we have 
\begin{align}
    \|\varphi_{m,l}(1_{B_l})\varphi_{m,n}\big(\varphi_{n,k}(\hat{x})\varphi_{n,k}(\hat{y})\big)-\varphi_{m,k}(\hat{x})\varphi_{m,k}(\hat{y})\|&<\varepsilon. \label{eq: NF}
\end{align}
Expanding these terms, we get
\begin{align*}
\varphi_{m,k}(\hat{x})\varphi_{m,k}(\hat{y})&=\big(\hat{x} \oplus \big(\oplus_{j=k+1}^m \rho_{j,k}(x)\big)\big)\big(\hat{y} \oplus \big(\oplus_{j=k+1}^m \rho_{j,k}(y)\big)\big)\\
&=\hat{x}\hat{y} \oplus \big(\oplus_{j=k+1}^m \rho_{j,k}(x)\rho_{j,k}(y)\big),\\
    \varphi_{m,l}(1_{B_l})&=1_{B_l}\oplus \big(\oplus_{j=l+1}^m \rho_{j,l}(1_{F_l})\big),
    \end{align*}
    and
    \begin{align*}
    &\varphi_{m,n}(\varphi_{n,k}(\hat{x})\varphi_{n,k}(\hat{y}))\\&= \varphi_{m,n}\big(\hat{x}\hat{y} \oplus \big(\oplus_{j=k+1}^n \rho_{j,k}(x)\rho_{j,k}(y)\big)\big)\\
   & =\hat{x}\hat{y} \oplus \big(\oplus_{j=k+1}^n \rho_{j,k}(x)\rho_{j,k}(y)\big) \oplus \big(\oplus_{j=n+1}^m \rho_{j,k}\big(\rho_{n,k}(x)\rho_{n,k}(y)\big)\big).
\end{align*} 
By considering only the $m^{\text{th}}$ entries of these, it follows from \eqref{eq: NF} that
\begin{align*}
    \|\rho_{m,n}(1_{F_n})\rho_{m,n}\big(\rho_{n,k}(x)\rho_{n,k}(y)\big)-\rho_{m,k}(x)\rho_{m,k}(y)\|<\varepsilon, 
\end{align*}
which shows that $(F_n,\rho_{n+1,n})_n$ is a $\CPCstar$-system.
\end{proof}

%==============================================================================================================

\bibliographystyle{plain}
\makeatletter\renewcommand\@biblabel[1]{[#1]}\makeatother
\bibliography{References2}

\end{document}